\documentclass{amsart}

\usepackage[english]{babel}
\usepackage[letterpaper,top=2cm,bottom=2cm,left=3cm,right=3cm,marginparwidth=1.75cm]{geometry}

\usepackage[colorlinks=true, allcolors=blue]{hyperref}

\usepackage{quiver}
\usepackage{graphicx} 
\usepackage{amsthm,amsmath,amssymb,stmaryrd,enumerate,euscript,mathrsfs,a4,array}
\usepackage{multirow}
\usepackage{tikz-cd}
\usepackage{tikz}
\usepackage{cleveref}
\usepackage{cancel}

\newtheorem{theorem}{Theorem}[section]
\newtheorem*{theorem*}{Theorem}
\newtheorem{prop}[theorem]{Proposition}

\newtheorem{lemma}[theorem]{Lemma}

\theoremstyle{definition} 
\newtheorem{defn}[theorem]{Definition}

\theoremstyle{remark}
\newtheorem{rmk}[theorem]{Remark}
\newtheorem{eg}[theorem]{Example}

\newtheorem*{remark*}{Remark}
\newtheorem*{remarks*}{Remarks}
\newtheorem*{notation*}{Notation}
\newtheorem*{convention*}{Convention}

\def\red{\color{red}}

\def\black{\color{black}}

\newcommand{\gcat}{\mathcal}
\newcommand{\gk}{\gcat{K}}

\newcommand{\pschG}{\mathcal{G}}
\newcommand{\pschH}{\mathcal{H}}
\newcommand{\cc}{\mathcal{C}}
\newcommand{\labL}{\mathcal{L}}
\newcommand{\intchg}{\gamma}

\newcommand{\catname}[1]{{\normalfont\textbf{#1}}}

\newcommand{\Gray}{\catname{Gray}}

\title{Some Remarks on the Interchange in Gray-categories}
\author{Nicola Di Vittorio}
\author{Gabriele Lobbia}

\address{
Nicola \textsc{Di Vittorio}: \newline
School of Mathematical and Physical Sciences\newline
Macquarie University\newline
Sydney, Australia\newline
}

\address{
Gabriele Lobbia: \newline
Department of Mathematics and Statistics\newline
Masaryk University, Faculty of Sciences\newline
Kotl\'{a}\v{r}sk\'{a} 2, 611 37 Brno, Czech Republic
}

\begin{document}
\tikzset{Rightarrow/.style={double equal sign distance,>={Implies},->},
triple/.style={-,preaction={draw,Rightarrow}},
quadruple/.style={preaction={draw,Rightarrow,shorten >=0pt},shorten >=1pt,-,double,double
distance=0.2pt}}
\maketitle

\begin{abstract}
We prove a generalised interchange equality for 3-cells in a Gray-category, i.e.\ we show that it still holds modulo the unique isomorphism given by the Gray-categorical pasting theorem of Di Vittorio. This significantly simplifies many calculations in Gray-categories. 
\end{abstract}

\section{Introduction}

\subsection*{Pasting Diagrams}
Pasting diagrams are a graphical device used to express composition in higher dimensional categories, first introduced in \cite{benabou:introduction}. A 2-dimensional pasting diagram looks like the following picture
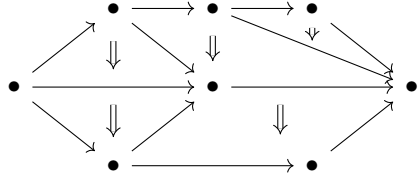
\[\begin{tikzcd}
	& \bullet & \bullet & \bullet \\
	\bullet && \bullet && \bullet \\
	& \bullet && \bullet
	\arrow[from=2-1, to=1-2]
	\arrow[from=1-2, to=1-3]
	\arrow[from=1-3, to=1-4]
	\arrow[from=1-4, to=2-5]
	\arrow[from=2-1, to=3-2]
	\arrow[from=3-2, to=3-4]
	\arrow[from=3-4, to=2-5]
	\arrow[""{name=0, anchor=center, inner sep=0}, from=2-1, to=2-3]
	\arrow[from=1-2, to=2-3]
	\arrow[""{name=1, anchor=center, inner sep=0}, from=2-3, to=2-5]
	\arrow[from=3-2, to=2-3]
	\arrow[""{name=2, anchor=center, inner sep=0}, from=1-3, to=2-5]
	\arrow[shorten <=4pt, shorten >=4pt, Rightarrow, from=1-3, to=2-3]
	\arrow[shorten <=6pt, shorten >=6pt, Rightarrow, from=1-2, to=0]
	\arrow[shorten <=1pt, shorten >=2pt, Rightarrow, from=1-4, to=2]
	\arrow[shorten <=6pt, shorten >=5pt, Rightarrow, from=0, to=3-2]
	\arrow[shift right=5, shorten <=6pt, shorten >=6pt, Rightarrow, from=1, to=3-4]
\end{tikzcd}\]
where the dots are thought as placeholders for objects in a higher dimensional category and the cells between them for higher cells. When considering 2-dimensional pasting diagrams, 
the composition we are actually interested in is the one between 2-cells, 
so that we end up with a 2-cell between the composition of the uppermost and the lowermost $n$-tuples of composable 1-cells. This idea can be generalised to higher dimensions. For instance, one can talk about 3-dimensional pasting diagrams with 3-cells having 2-dimensional pasting diagrams as domains and codomains and so on. The picture below is an example of one of these.
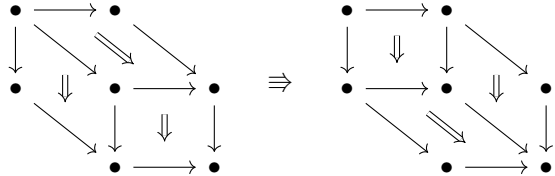
\begin{center}
\begin{tikzcd}[ampersand replacement=\&]
	\bullet \& \bullet \\
	\bullet \& \bullet \& \bullet \\
	\& \bullet \& \bullet
	\arrow[from=1-2, to=2-3]
	\arrow[""{name=0, anchor=center, inner sep=0}, from=1-1, to=2-2]
	\arrow[""{name=1, anchor=center, inner sep=0}, from=2-2, to=2-3]
	\arrow[from=2-3, to=3-3]
	\arrow[""{name=2, anchor=center, inner sep=0}, from=1-1, to=1-2]
	\arrow[from=2-2, to=3-2]
	\arrow[from=1-1, to=2-1]
	\arrow[""{name=3, anchor=center, inner sep=0}, from=2-1, to=3-2]
	\arrow[""{name=4, anchor=center, inner sep=0}, from=3-2, to=3-3]
	\arrow[shorten <=14pt, shorten >=14pt, Rightarrow, from=2, to=1]
	\arrow[shorten <=9pt, shorten >=9pt, Rightarrow, from=1, to=4]
	\arrow[shorten <=9pt, shorten >=9pt, Rightarrow, from=0, to=3]
\end{tikzcd}
    $\quad\Rrightarrow\quad$ 
\begin{tikzcd}[ampersand replacement=\&]
	\bullet \& \bullet \\
	\bullet \& \bullet \& \bullet \\
	\& \bullet \& \bullet
	\arrow[""{name=0, anchor=center, inner sep=0}, from=1-1, to=1-2]
	\arrow[from=1-2, to=2-2]
	\arrow[""{name=1, anchor=center, inner sep=0}, from=1-2, to=2-3]
	\arrow[""{name=2, anchor=center, inner sep=0}, from=2-2, to=3-3]
	\arrow[from=2-3, to=3-3]
	\arrow[from=1-1, to=2-1]
	\arrow[from=2-1, to=3-2]
	\arrow[""{name=3, anchor=center, inner sep=0}, from=3-2, to=3-3]
	\arrow[""{name=4, anchor=center, inner sep=0}, from=2-1, to=2-2]
	\arrow[shorten <=9pt, shorten >=9pt, Rightarrow, from=0, to=4]
	\arrow[shorten <=14pt, shorten >=14pt, Rightarrow, from=4, to=3]
	\arrow[shorten <=9pt, shorten >=9pt, Rightarrow, from=1, to=2]
\end{tikzcd}
\end{center}
There are a number of ways to make the notion of pasting diagram precise. In the 2-dimensional setting it is often useful to work with \emph{pasting schemes} (see \cite{AJPow90} and \cite[Definition~4.6]{div:past-thm-gray}), i.e.\ directed graphs with the following properties:
\begin{enumerate}
    \item[(i)] they are \emph{plane}, meaning it comes with a specified embedding in $\mathbb{R}^2$ that preserves orientation,
    \item[(ii)] they have distinct vertices $s$ and $t$ such that for every vertex $v$ there exist directed paths from $s$ to $v$ and from $v$ to $t$,
    \item[(iii)] they are \emph{acyclic}, meaning there are no directed paths from a vertex to itself.
\end{enumerate}

An equivalent characterization of pasting schemes, provided in \cite{AJPow90}, ensures that for every interior face $F$ there exist distinct vertices $s_F$ and $t_F$ and directed paths $\sigma_F$
and $\tau_F$ from $s_F$ to $t_F$ such that the boundary of $F$ is the directed path $\sigma_F\tau_F^*$, where by $\tau_F^*$ we mean that $\tau_F$ is traversed in the opposite way. We refer to $\sigma_F$ and $\tau_F$ as the uppermost and lowermost paths of $F$. 

\subsection*{Interchange} A 2-category consists of objects, morphisms between objects (1-cells) and morphisms between morphisms (2-cells) satisfying associativity and unitality axioms. One of the most important equality that follows from this axioms is the so called \emph{interchange law}, which states that for any pair of 2-cells $\alpha$ and $\beta$ of the form shown below left, we can apply them to the composite 1-cell $gf$ in any order and the resulting composite 2-cell will not change. This condition is expressed precisely by the commutativity of the diagram below right. 
\begin{equation}\label{eq:interchange-2-cats}
\begin{tikzcd}[ampersand replacement=\&]
	X \& Y \& Z
	\arrow[""{name=0, anchor=center, inner sep=0}, "f", curve={height=-12pt}, from=1-1, to=1-2]
	\arrow[""{name=1, anchor=center, inner sep=0}, "{f'}"', curve={height=12pt}, from=1-1, to=1-2]
	\arrow[""{name=2, anchor=center, inner sep=0}, "g", curve={height=-12pt}, from=1-2, to=1-3]
	\arrow[""{name=3, anchor=center, inner sep=0}, "{g'}"', curve={height=12pt}, from=1-2, to=1-3]
	\arrow["\alpha"{xshift=0.1cm}, shorten <=6pt, shorten >=6pt, Rightarrow, from=0, to=1]
	\arrow["\beta"{xshift=0.1cm}, shorten <=6pt, shorten >=6pt, Rightarrow, from=2, to=3]
\end{tikzcd} \hspace{2cm}
\begin{tikzcd}[ampersand replacement=\&]
	gf \& {gf'} \\
	{g'f} \& {g'f'}
	\arrow["1\circ\alpha", from=1-1, to=1-2]
	\arrow["{\beta\circ 1}"', from=1-1, to=2-1]
	\arrow["1\circ\alpha"', from=2-1, to=2-2]
	\arrow["\beta\circ1", from=1-2, to=2-2]
\end{tikzcd}  
\end{equation}
Using these axioms, and in particular the interchange law, it is possible to prove that all pasting diagrams in a 2-category have a unique composite \cite{AJPow90}. As a consequence, all calculations involving 2-cells in a 2-category can be done using 2-dimensional pasting diagrams. Unfortunately, increasing the dimension and considering weaker notions of higher categories this is not necessarily true, see for instance \cite{div:past-thm-gray}. In this paper we will focus on the 3-dimensional scenario, building on \cite{div:past-thm-gray}. We will show that, up to unique isomorphism, one can actually use 3-dimensional pasting diagrams for computations involving 3-cells in a Gray-category.  

\subsection*{Gray-categories}

In 3-dimensional category theory it is quite useful to work with Gray-categories, i.e.\ semistrict tricategories \cite[Section~4.8]{GordonR:coht}. These are easier to handle than tricategories, but sufficiently general for many purposes, since every tricategory is triequivalent to a Gray-category  \cite[Theorem~8.1]{GordonR:coht}. 

A Gray-category is a category enriched over the monoidal category of $2$-categories and strict $2$-functors equipped with the Gray tensor product, see  \cite[Section~3.2]{GurskiN:algtt} or \cite[Section~4]{GordonR:coht}. Unpacking this definition, a $\Gray$-category $\gk$ consists of the following data:
\begin{enumerate}[i)]
    \item a class of objects $\gk_0$,
    \item for each couple of objects $X$ and $Y$ in $\gk$, a $2$-category $\gk(X,Y)$,
    \item for every $X\in\gk$ an identity $1$-cell $1_X\colon X\to X$,
    \item a composition $2$-functor $c_{X,Y,Z}\colon\gk(Y,Z)\otimes\gk(X,Y)\to\gk(X,Z)$ satisfying associativity and unitality rules, where $\otimes$ here stands for the Gray-tensor product of 2-categories.
\end{enumerate}

The non-strictness of Gray-categories is a consequence of the definition of the Gray tensor product. In particular, this implies that, given 2-cells as in \eqref{eq:interchange-2-cats}, then instead of an equality there is an invertible 3-cell between the two composites, usually referred to as the \emph{interchange isomorphism} in a Gray-category. 
\begin{center}
\begin{tikzcd}[ampersand replacement=\&]
	gf \& {gf'} \\
	{g'f} \& {g'f'}
	\arrow[""{name=0, anchor=center, inner sep=0}, "1\circ\alpha", from=1-1, to=1-2]
	\arrow["{\beta\circ 1}"', from=1-1, to=2-1]
	\arrow[""{name=1, anchor=center, inner sep=0}, "1\circ\alpha"', from=2-1, to=2-2]
	\arrow["\beta\circ1", from=1-2, to=2-2]
	\arrow["{\gamma_{\beta,\alpha}}", shorten <=9pt, shorten >=9pt, Rightarrow, from=0, to=1]
\end{tikzcd} \hspace{0.5cm} $\in\gk(X,Z)$
\end{center}
Another way of saying this is that the interchange, for 2-cells, holds only up to coherent isomorphism. The interchange maps satisfy several equations, for which we refer the reader to \cite[Remark~1.2]{GambLob:psm}. We recall one of them below since it will be useful in \Cref{thm:easy-nat-3-cell}. For 1-cells $f\colon X\to Y$ and $g\colon Y\to Z$, 2-cells $\alpha,\alpha'\colon f\to f'$ and $\beta,\beta'\colon g\to g'$, and 3-cells $\Gamma\colon\alpha\to\alpha'$ and $\Delta\colon\beta\to\beta'$, the following equation holds. 
\begin{center}
\begin{tikzcd}[ampersand replacement=\&]
	gf \&\& {gf'} \\
	\\
	{g'f} \&\& {g'f'}
	\arrow[""{name=0, anchor=center, inner sep=0}, "1\circ\alpha"{description}, curve={height=-12pt}, from=1-1, to=1-3]
	\arrow[""{name=1, anchor=center, inner sep=0}, "{\beta'\circ 1}"', curve={height=12pt}, from=1-1, to=3-1]
	\arrow["{1\circ\alpha'}"', from=3-1, to=3-3]
	\arrow["\beta\circ1", from=1-3, to=3-3]
	\arrow[""{name=2, anchor=center, inner sep=0}, "{1\circ\alpha'}"{description}, curve={height=12pt}, from=1-1, to=1-3]
	\arrow["{\gamma_{\beta,\alpha'}}", shift left=2, shorten <=22pt, shorten >=22pt, Rightarrow, from=1-3, to=3-1]
	\arrow[""{name=3, anchor=center, inner sep=0}, "\beta\circ1"{description, pos=0.7}, curve={height=-12pt}, from=1-1, to=3-1]
	\arrow["1\circ\Gamma"{xshift=0.1cm}, shift right=3, shorten <=6pt, shorten >=6pt, Rightarrow, from=0, to=2]
	\arrow["\Delta\circ1"'{yshift=0.1cm}, shorten <=6pt, shorten >=6pt, Rightarrow, from=3, to=1]
\end{tikzcd} $=$
\begin{tikzcd}[ampersand replacement=\&]
	gf \&\& {gf'} \\
	\\
	{g'f} \&\& {g'f'}
	\arrow["1\circ\alpha"{description}, from=1-1, to=1-3]
	\arrow["{\beta'\circ 1}"', from=1-1, to=3-1]
	\arrow[""{name=0, anchor=center, inner sep=0}, "{1\circ\alpha'}"', curve={height=12pt}, from=3-1, to=3-3]
	\arrow[""{name=1, anchor=center, inner sep=0}, "\beta\circ1", curve={height=-12pt}, from=1-3, to=3-3]
	\arrow[""{name=2, anchor=center, inner sep=0}, "1\circ\alpha"{description}, curve={height=-12pt}, from=3-1, to=3-3]
	\arrow["{\gamma_{\beta',\alpha}}"', shorten <=22pt, shorten >=22pt, Rightarrow, from=1-3, to=3-1]
	\arrow[""{name=3, anchor=center, inner sep=0}, "{\beta'\circ 1}"{description, pos=0.3}, curve={height=12pt}, from=1-3, to=3-3]
	\arrow["1\circ\Gamma"{xshift=0.1cm}, shift right=3, shorten <=6pt, shorten >=6pt, Rightarrow, from=2, to=0]
	\arrow["\Delta\circ1"{yshift=-0.1cm}, shorten <=6pt, shorten >=6pt, Rightarrow, from=1, to=3]
\end{tikzcd}
\end{center}
Equivalently, we can divide this axiom in the two equations below. 
\begin{equation}\label{eq:interchange-1st-comp}
\begin{tikzcd}[ampersand replacement=\&]
	gf \&\& {gf'} \\
	{g'f} \&\& {g'f'}
	\arrow[""{name=0, anchor=center, inner sep=0}, "1\circ\alpha"{description}, curve={height=-12pt}, from=1-1, to=1-3]
	\arrow[""{name=1, anchor=center, inner sep=0}, "{1\circ\alpha'}"', curve={height=12pt}, from=2-1, to=2-3]
	\arrow["\beta\circ1", from=1-3, to=2-3]
	\arrow[""{name=2, anchor=center, inner sep=0}, "{1\circ\alpha'}"{description}, curve={height=12pt}, from=1-1, to=1-3]
	\arrow["\beta\circ1"', from=1-1, to=2-1]
	\arrow["1\circ\Gamma", shift right=3, shorten <=6pt, shorten >=6pt, Rightarrow, from=0, to=2]
	\arrow["{\gamma_{\beta,\alpha'}}", shorten <=9pt, shorten >=9pt, Rightarrow, from=2, to=1]
\end{tikzcd} =
\begin{tikzcd}[ampersand replacement=\&]
	gf \&\& {gf'} \\
	{g'f} \&\& {g'f'}
	\arrow[""{name=0, anchor=center, inner sep=0}, "1\circ\alpha"{description}, curve={height=-12pt}, from=1-1, to=1-3]
	\arrow["{\beta\circ 1}"', from=1-1, to=2-1]
	\arrow[""{name=1, anchor=center, inner sep=0}, "{1\circ\alpha'}"', curve={height=12pt}, from=2-1, to=2-3]
	\arrow[""{name=2, anchor=center, inner sep=0}, "1\circ\alpha"{description}, curve={height=-12pt}, from=2-1, to=2-3]
	\arrow["{\beta\circ 1}", from=1-3, to=2-3]
	\arrow["1\circ\Gamma", shift right=3, shorten <=6pt, shorten >=6pt, Rightarrow, from=2, to=1]
	\arrow["{\gamma_{\beta,\alpha}}", shorten <=9pt, shorten >=9pt, Rightarrow, from=0, to=2]
\end{tikzcd}
\end{equation}
\begin{equation}\label{eq:interchange-2nd-comp}
\begin{tikzcd}[ampersand replacement=\&]
	gf \&\& {gf'} \\
	{g'f} \&\& {g'f'}
	\arrow["1\circ\alpha", from=1-1, to=1-3]
	\arrow[""{name=0, anchor=center, inner sep=0}, "{\beta'\circ 1}"', curve={height=12pt}, from=1-1, to=2-1]
	\arrow["1\circ\alpha"', from=2-1, to=2-3]
	\arrow[""{name=1, anchor=center, inner sep=0}, "\beta\circ1", curve={height=-12pt}, from=1-3, to=2-3]
	\arrow[""{name=2, anchor=center, inner sep=0}, "{\beta'\circ 1}"'{pos=0.6}, curve={height=12pt}, from=1-3, to=2-3]
	\arrow["\Delta\circ1"'{yshift=0.1cm}, shift left=2, shorten <=5pt, shorten >=5pt, Rightarrow, from=1, to=2]
	\arrow["{\gamma_{\beta',\alpha}}"'{yshift=0.1cm}, shorten <=26pt, shorten >=26pt, Rightarrow, from=2, to=0]
\end{tikzcd} =
\begin{tikzcd}[ampersand replacement=\&]
	gf \&\& {gf'} \\
	{g'f} \&\& {g'f'}
	\arrow["1\circ\alpha"{description}, from=1-1, to=1-3]
	\arrow[""{name=0, anchor=center, inner sep=0}, "{\beta'\circ 1}"', curve={height=12pt}, from=1-1, to=2-1]
	\arrow[""{name=1, anchor=center, inner sep=0}, "\beta\circ1", curve={height=-12pt}, from=1-3, to=2-3]
	\arrow[""{name=2, anchor=center, inner sep=0}, "\beta\circ1"{pos=0.6}, curve={height=-12pt}, from=1-1, to=2-1]
	\arrow["1\circ\alpha"', from=2-1, to=2-3]
	\arrow["\Delta\circ1"'{yshift=0.1cm}, shift left, shorten <=5pt, shorten >=5pt, Rightarrow, from=2, to=0]
	\arrow["{\gamma_{\beta,\alpha}}"'{yshift=0.1cm}, shorten <=26pt, shorten >=26pt, Rightarrow, from=1, to=2]
\end{tikzcd}
\end{equation}

On the other hand, a Gray-category has hom-2-categories, so the interchange for 3-cells holds strictly. Precisely, for 3-cells $\Gamma\colon\alpha\Rrightarrow\alpha'$ and $\Delta\colon\beta\Rrightarrow\beta'$ as below left, the diagram below right commutes. 
\begin{center}
\begin{tabular}{cc|c}
    In $\gk$ &  In $\gk(X,Y)$ & \emph{Local Interchange}\\
\begin{tikzcd}[ampersand replacement=\&]
	X \&\& Y
	\arrow[""{name=0, anchor=center, inner sep=0}, "{f_1}", curve={height=-30pt}, from=1-1, to=1-3]
	\arrow[""{name=1, anchor=center, inner sep=0}, "{f_2}"{description}, from=1-1, to=1-3]
	\arrow[""{name=2, anchor=center, inner sep=0}, "{f_3}"', curve={height=30pt}, from=1-1, to=1-3]
	\arrow[""{name=3, anchor=center, inner sep=0}, "{\alpha'}", shift left=4, shorten <=6pt, shorten >=6pt, Rightarrow, from=0, to=1]
	\arrow[""{name=4, anchor=center, inner sep=0}, "\alpha"', shift right=4, shorten <=6pt, shorten >=6pt, Rightarrow, from=0, to=1]
	\arrow[""{name=5, anchor=center, inner sep=0}, "\beta"', shift right=4, shorten <=6pt, shorten >=6pt, Rightarrow, from=1, to=2]
	\arrow[""{name=6, anchor=center, inner sep=0}, "{\beta'}", shift left=4, shorten <=6pt, shorten >=6pt, Rightarrow, from=1, to=2]
	\arrow["\Gamma", triple, shorten <=4pt, shorten >=4pt, from=4, to=3]
	\arrow["\Delta"', triple, shorten <=4pt, shorten >=4pt, from=5, to=6]
\end{tikzcd} & 
\begin{tikzcd}[ampersand replacement=\&]
	{f_1} \& {f_2} \& {f_3}
	\arrow[""{name=0, anchor=center, inner sep=0}, "\alpha", curve={height=-12pt}, from=1-1, to=1-2]
	\arrow[""{name=1, anchor=center, inner sep=0}, "{\alpha'}"', curve={height=12pt}, from=1-1, to=1-2]
	\arrow[""{name=2, anchor=center, inner sep=0}, "\beta", curve={height=-12pt}, from=1-2, to=1-3]
	\arrow[""{name=3, anchor=center, inner sep=0}, "{\beta'}"', curve={height=12pt}, from=1-2, to=1-3]
	\arrow["\Gamma"{xshift=0.1cm}, shift right=1, shorten <=6pt, shorten >=6pt, Rightarrow, from=0, to=1]
	\arrow["\Delta"{xshift=0.1cm}, shift right=1, shorten <=6pt, shorten >=6pt, Rightarrow, from=2, to=3]
\end{tikzcd} 
&
\begin{tikzcd}[ampersand replacement=\&]
	\beta\cdot\alpha \& {\beta\cdot\alpha'} \\
	{\beta'\cdot\alpha} \& {\beta'\cdot\alpha'}
	\arrow["{1\cdot \Gamma}", from=1-1, to=1-2]
	\arrow["\Delta\cdot1"', from=1-1, to=2-1]
	\arrow["{1\cdot \Gamma}"', from=2-1, to=2-2]
	\arrow["\Delta\cdot1", from=1-2, to=2-2]
\end{tikzcd}
\end{tabular}
\end{center}
We will refer to this instance of the interchange as \emph{local interchange}. 

\subsection*{Gray-categorical Pasting Theorem}

Let $\gk$ be a Gray-category and $\pschG$ a pasting scheme. 
We recall that a \emph{labelling} $\labL$ of $\pschG$ in $\gk$ \cite[Definition~4.22]{div:past-thm-gray} is an assignment of a 0-cell in $\gk$ to each vertex of $\pschG$, a 1-cell to each edge and a 2-cell to each face of $\pschG$, in a way that preserves domains and codomains. \black 
In \cite[Lemma~4.23]{div:past-thm-gray} we get a canonical functor from the groupoid $\cc_\pschG$,
which encodes the different orders of composition of the pasting diagram given by $\labL$, \black 
$$\labL\colon \cc_\pschG\to\gk(S,T)[p,q]$$
where $S,T$ are the source and sink of $\pschG$ and $p,q$ the top and bottom paths. 
The idea is that this functor takes as input a possible order of composition and gives out the composite 2-cell in $\gk$. \black We will abuse notation by often referring to both the labelling and the canonical functor with the same letter $\labL$. 

In \cite[Theorem~4.24]{div:past-thm-gray} it is proven that there exists a unique isomorphism between two different possible composition of a labelling of a pasting scheme. We might refer to this with
$$\intchg_{X,Y}=\intchg\colon X=\labL(\overline{F})\to\labL(\overline{G})=Y$$
where $\overline{F}$ and $\overline{G}$ are the objects of $\cc_\pschG$ corresponding to the compositions. 
The aim of this paper is to show that the local interchange in a Gray-category is still true even if we change the composite of a pasting diagram at any point (\Cref{thm:easy-nat-3-cell,thm:easy-gen-loc-int}). Moreover, we will show how to apply this to the specific case of 3-cells acting on more than one 2-cell in a pasting diagram (\Cref{thm:hard-nat-3-cell,thm:hard-gen-loc-int}).

\subsection*{Applications} In addition to the key role played by Gray-categories in higher category theory, thanks to strictification results such as the classic \cite{GordonR:coht}, they are also important in homotopy theory since Gray-groupoids (i.e.\ Gray-categories where every $n$-morphism is invertible for $n=1,2,3$) model homotopy 3-types \cite{lack:gray_quillen}. Moreover, Gray-categories found applications also in mathematical physics, in particular in TQFT \cite{BMS_gray-duals,CARQUEVILLE2020107024,CarMue:orb-comp-3-cat} and higher gauge theory \cite{FMarPick:fund-gray-grp,wang20143}. The fail of the interchange leads to complications in the various calculations one has to perform when dealing with Gray-categories. This problem was addressed using Gray-categories with duals \cite[Definition~3.10]{BMS_gray-duals} with a \emph{geometric} approach. This paper provides a new approach to simplify calculations in Gray-categories, allowing one to work with 3-dimensional pasting diagrams. 


\section{1-to-1 Local Interchange}

In this section we will consider a labelling $\labL$ of a pasting scheme $\pschG$ in $\gk$, two 2-cells $\alpha$ and $\beta$ in the labelling corresponding to the faces $F_\alpha\neq F_\beta$ of the pasting scheme and two 3-cells $\Gamma\colon\alpha\Rrightarrow\alpha'$ and $\Delta\colon\beta\Rrightarrow\beta'$. 

\begin{notation*}
Given a labelling $\labL$ of a pasting scheme $\pschG$, we will denote with $\labL^{\alpha'}$ the labelling which has $\alpha'$ instead of $\alpha$ in the specified face $F_\alpha$ and coincides with $\labL$ everywhere else. 
\end{notation*}

\begin{rmk}
    Clearly given the notation above and two pairs of 2-cells $\alpha,\alpha'$ and $\beta,\beta'$, then
    $$(\labL^{\alpha'})^{\beta'}=(\labL^{\beta'})^{\alpha'}=:\labL^{\alpha',\beta'}.$$
\end{rmk}

\begin{theorem}\label{thm:easy-nat-3-cell}
    Any 3-cell $\Gamma\colon\alpha\Rrightarrow\alpha'$ induces a natural transformation
\[\begin{tikzcd}[ampersand replacement=\&]
	{\cc_\pschG} \&\& {\gk(S,T)[p,q].}
	\arrow[""{name=0, anchor=center, inner sep=0}, "\labL", shift left=1, curve={height=-15pt}, from=1-1, to=1-3]
	\arrow[""{name=1, anchor=center, inner sep=0}, "{\labL^{\alpha'}}"', shift right=1, curve={height=15pt}, from=1-1, to=1-3]
	\arrow["{\Gamma^\ast}"{xshift=0.1cm}, shorten <=15pt, shorten >=15pt, Rightarrow, from=0, to=1]
\end{tikzcd}\]
\end{theorem}

\begin{proof}
  For an object $\overline{F}=\overline{H}F_\alpha\overline{K}\in\cc_\pschG$,
  we define the component of $\Gamma^\ast$ at $\overline{F}$ to be the morphism  $\Gamma^\ast_{\overline{F}}\colon\labL(\overline{F})\to\labL^{\alpha'}(\overline{F})$ obtained by whiskering $\Gamma$ with the labelling of the faces in $\overline{H}$ and $\overline{K}$. The diagram below, which is in the 2-category $\gk(S,T)$, is a sketch of this map, where we write $\overline{(-)}$ for the appropriate whiskering with 1-cells. 

\[\begin{tikzcd}[ampersand replacement=\&]
	\bullet \& \bullet \& \bullet \& \bullet \& \bullet \& \bullet
	\arrow["{\overline{\alpha_1}}"{description}, from=1-1, to=1-2]
	\arrow["{\overline{\alpha_n}}"{description}, from=1-5, to=1-6]
	\arrow[""{name=0, anchor=center, inner sep=0}, "{\labL(\overline{F})}", curve={height=-50pt}, from=1-1, to=1-6]
	\arrow[""{name=1, anchor=center, inner sep=0}, "{\labL^{\alpha'}(\overline{F})}"', curve={height=50pt}, from=1-1, to=1-6]
	\arrow[""{name=2, anchor=center, inner sep=0}, "{\overline{\alpha}}", curve={height=-12pt}, from=1-3, to=1-4]
	\arrow[""{name=3, anchor=center, inner sep=0}, "{\overline{\alpha'}}"', curve={height=12pt}, from=1-3, to=1-4]
	\arrow["\ldots"{description}, draw=none, from=1-2, to=1-3]
	\arrow["\ldots"{description}, draw=none, from=1-4, to=1-5]
	\arrow["{:=}"{marking,xshift=-0.2cm}, draw=none, from=0, to=2]
	\arrow["{\overline{\Gamma}}"{xshift=0.1cm},shift right=1, shorten <=6pt, shorten >=6pt, Rightarrow, from=2, to=3]
	\arrow["{:=}"{marking,xshift=-0.2cm}, draw=none, from=1, to=3]
\end{tikzcd}\]
  The naturality of $\Gamma^\ast$ follows from the coherence equations for the interchange, as we will show next. The morphisms in $\cc_\pschG$ are generated by $\overline{U}\widehat{GH}\overline{V}\colon \overline{U}GH\overline{V}\to\overline{U}HG\overline{V}$, so it is enough to show naturality for these maps, i.e. 
\[\begin{tikzcd}[ampersand replacement=\&]
	{\labL(\overline{U}GH\overline{V})} \& {\labL^{\alpha'}(\overline{U}GH\overline{V})} \\
	{\labL(\overline{U}HG\overline{V})} \& {\labL^{\alpha'}(\overline{U}GH\overline{V}).}
	\arrow["\gamma"', from=1-1, to=2-1]
	\arrow["\gamma", from=1-2, to=2-2]
	\arrow["{\Gamma^\ast}"', from=2-1, to=2-2]
	\arrow["{\Gamma^\ast}", from=1-1, to=1-2]
\end{tikzcd}\]
Let us write $\sigma=\labL(G)$ and $\delta=\labL(H)$. All cases can be reduced to the following diagram.
\[\begin{tikzcd}[ampersand replacement=\&]
	S \& A \& B \& C \& D \& T
	\arrow["f"{description}, from=1-1, to=1-2]
	\arrow[""{name=0, anchor=center, inner sep=0}, "{a}", curve={height=-12pt}, from=1-2, to=1-3]
	\arrow[""{name=1, anchor=center, inner sep=0}, "{a'}"', curve={height=12pt}, from=1-2, to=1-3]
	\arrow[""{name=2, anchor=center, inner sep=0}, "g"{description}, from=1-3, to=1-4]
	\arrow[""{name=3, anchor=center, inner sep=0}, "{b}", curve={height=-12pt}, from=1-4, to=1-5]
	\arrow[""{name=4, anchor=center, inner sep=0}, "{b'}"', curve={height=12pt}, from=1-4, to=1-5]
	\arrow["h"{description}, from=1-5, to=1-6]
	\arrow[""{name=5, anchor=center, inner sep=0}, "p", curve={height=-40pt}, from=1-1, to=1-6]
	\arrow[""{name=6, anchor=center, inner sep=0}, "q"', curve={height=40pt}, from=1-1, to=1-6]
	\arrow["\phi", shorten <=8pt, shorten >=8pt, Rightarrow, from=5, to=2]
	\arrow["\psi", shorten <=8pt, shorten >=8pt, Rightarrow, from=2, to=6]
	\arrow["\sigma", shorten <=6pt, shorten >=6pt, Rightarrow, from=0, to=1]
	\arrow["\delta", shorten <=6pt, shorten >=6pt, Rightarrow, from=3, to=4]
\end{tikzcd}\]

  \begin{itemize}
      \item If $F_\alpha=G$, and so $\alpha=\sigma$, the wanted equality becomes the equality between the following pasting diagrams in the 2-category $\gk(S,T)$.
\begin{center}
\scalebox{0.9}{
\hspace{-1.5cm}
\begin{tikzcd}[ampersand replacement=\&]
	p \& hbgaf \&\& {hbga'f} \\
	\& {hb'gaf} \&\& {hb'ga'f} \& q
	\arrow["\psi"', from=2-4, to=2-5]
	\arrow["\phi", from=1-1, to=1-2]
	\arrow[""{name=0, anchor=center, inner sep=0}, "{1\circ\alpha'\circ1}"{description}, curve={height=12pt}, from=1-2, to=1-4]
	\arrow[""{name=1, anchor=center, inner sep=0}, "1\circ\alpha\circ1", curve={height=-12pt}, from=1-2, to=1-4]
	\arrow["1\circ\delta\circ1", from=1-4, to=2-4]
	\arrow["1\circ\delta\circ1"', from=1-2, to=2-2]
	\arrow[""{name=2, anchor=center, inner sep=0}, "{1\circ\alpha'\circ1}"', curve={height=12pt}, from=2-2, to=2-4]
	\arrow["{1\circ \Gamma\circ 1}", shift right=4, shorten <=6pt, shorten >=6pt, Rightarrow, from=1, to=0]
	\arrow["{\gamma_{\delta g,\alpha'}}", shorten <=9pt, shorten >=9pt, Rightarrow, from=0, to=2]
\end{tikzcd}$=$
\begin{tikzcd}[ampersand replacement=\&]
	p \& hbgaf \&\& {hbga'f} \\
	\& {hb'gaf} \&\& {hb'ga'f} \& q
	\arrow["\psi"', from=2-4, to=2-5]
	\arrow["\phi", from=1-1, to=1-2]
	\arrow[""{name=0, anchor=center, inner sep=0}, "1\circ\alpha\circ1", curve={height=-12pt}, from=1-2, to=1-4]
	\arrow["1\circ\delta\circ1", from=1-4, to=2-4]
	\arrow["1\circ\delta\circ1"', from=1-2, to=2-2]
	\arrow[""{name=1, anchor=center, inner sep=0}, "{1\circ\alpha'\circ1}"', curve={height=12pt}, from=2-2, to=2-4]
	\arrow[""{name=2, anchor=center, inner sep=0}, "1\circ\alpha\circ1"{description}, curve={height=-12pt}, from=2-2, to=2-4]
	\arrow["{1\circ \Gamma\circ 1}", shift right=3, shorten <=6pt, shorten >=6pt, Rightarrow, from=2, to=1]
	\arrow["{\gamma_{\delta g,\alpha}}", shorten <=9pt, shorten >=9pt, Rightarrow, from=0, to=2]
\end{tikzcd} 
}
\end{center} 
 This is true by axiom \eqref{eq:interchange-1st-comp} of the interchange isomorphism in the Gray-category $\gk$.
      \item If $F_\alpha=H$, and so $\alpha=\delta$, the wanted equality becomes
\begin{center}
\scalebox{0.9}{
\hspace{-1.5cm}
\begin{tikzcd}[ampersand replacement=\&]
	p \& hbgaf \& {hbga'f} \\
	\\
	\& {hb'gaf} \& {hb'ga'f} \& q
	\arrow["\rho"', from=3-3, to=3-4]
	\arrow["\phi", from=1-1, to=1-2]
	\arrow["1\circ\sigma\circ1", from=1-2, to=1-3]
	\arrow[""{name=0, anchor=center, inner sep=0}, "{1\circ\alpha'\circ1}"{description, pos=0.7}, curve={height=15pt}, from=1-3, to=3-3]
	\arrow[""{name=1, anchor=center, inner sep=0}, "{1\circ\alpha'\circ1}"', curve={height=15pt}, from=1-2, to=3-2]
	\arrow["1\circ\sigma\circ1"', from=3-2, to=3-3]
	\arrow[""{name=2, anchor=center, inner sep=0}, "1\circ\alpha\circ1", curve={height=-15pt}, from=1-3, to=3-3]
	\arrow["{\gamma_{\alpha',g\sigma}}"'{yshift=0.1cm}, shorten <=20pt, shorten >=20pt, Rightarrow, from=0, to=1]
	\arrow["1\circ\Gamma\circ1"'{yshift=0.1cm}, shift left, shorten <=6pt, shorten >=6pt, Rightarrow, from=2, to=0]
\end{tikzcd} $=$
\begin{tikzcd}[ampersand replacement=\&]
	p \& hbgaf \& {hbga'f} \\
	\\
	\& {hb'gaf} \& {hb'ga'f} \& q
	\arrow["\rho"', from=3-3, to=3-4]
	\arrow["\phi", from=1-1, to=1-2]
	\arrow["1\circ\sigma\circ1", from=1-2, to=1-3]
	\arrow[""{name=0, anchor=center, inner sep=0}, "{1\circ\alpha'\circ1}"', curve={height=15pt}, from=1-2, to=3-2]
	\arrow["1\circ\sigma\circ1"', from=3-2, to=3-3]
	\arrow[""{name=1, anchor=center, inner sep=0}, "1\circ\alpha\circ1"{pos=0.6}, curve={height=-15pt}, from=1-3, to=3-3]
	\arrow[""{name=2, anchor=center, inner sep=0}, "1\circ\alpha\circ1"{description, pos=0.7}, curve={height=-15pt}, from=1-2, to=3-2]
	\arrow["{\gamma_{\alpha,g\sigma}}"'{yshift=0.1cm}, shorten <=20pt, shorten >=20pt, Rightarrow, from=1, to=2]
	\arrow["1\circ\Gamma\circ1"'{yshift=0.1cm}, shorten <=6pt, shorten >=6pt, Rightarrow, from=2, to=0]
\end{tikzcd}
}
\end{center}
which is true by axiom \eqref{eq:interchange-2nd-comp} of the interchange isomorphism $\gamma$ in the Gray-category $\gk$. 

      \item If $F_\alpha\neq G$ and $F_\alpha\neq H$, then $F_\alpha$ is either in $\overline{U}$ or $\overline{V}$. The two cases are completely analogous, so let us assume $F_\alpha\in\overline{U}$. So we assume $\alpha=\phi$ and the equality is true because both correspond to the following pasting diagram in $\gk(S,T)$. 
\[
\scalebox{0.9}{\begin{tikzcd}[ampersand replacement=\&]
	\&\&\& {hbga'f} \\
	p \&\& hbgaf \&\& {hb'ga'f} \& q \\
	\&\&\& {hb'gaf}
	\arrow["\psi"', from=2-5, to=2-6]
	\arrow[""{name=0, anchor=center, inner sep=0}, "\alpha", curve={height=-20pt}, from=2-1, to=2-3]
	\arrow["1\circ\sigma\circ1", from=2-3, to=1-4]
	\arrow["1\circ\delta\circ1", from=1-4, to=2-5]
	\arrow["1\circ\delta\circ1"', from=2-3, to=3-4]
	\arrow["1\circ\sigma\circ1"', from=3-4, to=2-5]
	\arrow[""{name=1, anchor=center, inner sep=0}, "{\alpha'}"', curve={height=20pt}, from=2-1, to=2-3]
	\arrow["{\gamma_{\delta g,\alpha}}", shorten <=13pt, shorten >=13pt, Rightarrow, from=1-4, to=3-4]
	\arrow["\Gamma", shorten <=6pt, shorten >=6pt, Rightarrow, from=0, to=1]
\end{tikzcd}}\]
      
  \end{itemize}
\end{proof}

This theorem is saying that we can apply $\Gamma$ to any composite of the labelling and furthermore, given two composites $X$ and $X'$ we can either apply $\Gamma$ and then the unique interchange isomorphisms or the other way around. 
\[\begin{tikzcd}[ampersand replacement=\&]
	X \& {X'} \\
	Y \& {Y'}
	\arrow["\gamma", from=1-1, to=1-2]
	\arrow["{\Gamma^\ast}"', from=1-1, to=2-1]
	\arrow["\gamma"', from=2-1, to=2-2]
	\arrow["{\Gamma^\ast}", from=1-2, to=2-2]
\end{tikzcd}\]
The following theorem gives a first \emph{generalised local interchange} result for Gray-categories.
\begin{theorem}\label{thm:easy-gen-loc-int}
    Given two 3-cells $\Gamma\colon\alpha\Rrightarrow\alpha'$ and $\Delta\colon\beta\Rrightarrow\beta'$, then

\begin{center}

\scalebox{0.9}{
\begin{tikzcd}[ampersand replacement=\&]
	\& {} \\
	{\cc_{\pschG}} \&\& {\gk(S,T)[p,q]} \\
	\& {}
	\arrow["\labL", curve={height=-24pt}, from=2-1, to=2-3]
	\arrow[""{name=0, anchor=center, inner sep=0}, "{\labL^{\alpha'}}"{description}, from=2-1, to=2-3]
	\arrow["{\labL^{\alpha',\beta'}}"', curve={height=24pt}, from=2-1, to=2-3]
	\arrow["{\Gamma^\ast}"{xshift=0.1cm,yshift=-0.1cm}, shorten <=12pt, shorten >=8pt, Rightarrow, from=1-2, to=0]
	\arrow["{\Delta^\ast}"{xshift=0.1cm,yshift=0.1cm}, shorten <=8pt, shorten >=12pt, Rightarrow, from=0, to=3-2]
\end{tikzcd} $=$
\begin{tikzcd}[ampersand replacement=\&]
	\& {} \\
	{\cc_{\pschG}} \&\& {\gk(S,T)[p,q]} \\
	\& {}
	\arrow["\labL", curve={height=-24pt}, from=2-1, to=2-3]
	\arrow[""{name=0, anchor=center, inner sep=0}, "{\labL^{\beta'}}"{description}, from=2-1, to=2-3]
	\arrow["{\labL^{\alpha',\beta'}}"', curve={height=24pt}, from=2-1, to=2-3]
	\arrow["{\Delta^\ast}"{xshift=0.1cm,yshift=-0.1cm}, shorten <=12pt, shorten >=8pt, Rightarrow, from=1-2, to=0]
	\arrow["{\Gamma^\ast}"{xshift=0.1cm,yshift=0.1cm}, shorten <=8pt, shorten >=12pt, Rightarrow, from=0, to=3-2]
\end{tikzcd} 
}
\end{center}

\end{theorem}

\begin{proof}
    Let us consider an element $\overline{F}\in\cc_\pschG$ and let us assume, without loss of generality, that is of the form $\overline{F}=\overline{U}F_\alpha\overline{V}F_\beta\overline{W}$ (the case where $F_\beta$ appears before $F_\alpha$ is completely analogous). Before getting to the precise calculations, we will give some intuition of what is going on. Below we write $\labL(\overline{U})$ for the composite of the images of the 2-cells in the image of $\overline{U}$. Then, the equality we are looking for becomes the following equation. 
    \begin{center}
\begin{tikzcd}[ampersand replacement=\&]
	\bullet \& \bullet \& \bullet \& \bullet \& \bullet \& \bullet
	\arrow[""{name=0, anchor=center, inner sep=0}, "{\overline{\alpha}}", shift left, curve={height=-18pt}, from=1-2, to=1-3]
	\arrow[""{name=1, anchor=center, inner sep=0}, "{\overline{\alpha'}}"', shift right, from=1-2, to=1-3]
	\arrow["{\labL(\overline{U})}", from=1-1, to=1-2]
	\arrow["{\labL(\overline{V})}", from=1-3, to=1-4]
	\arrow["{\labL(\overline{W})}", from=1-5, to=1-6]
	\arrow[""{name=2, anchor=center, inner sep=0}, "{\overline{\beta'}}"', shift right, curve={height=18pt}, from=1-4, to=1-5]
	\arrow[""{name=3, anchor=center, inner sep=0}, "{\overline{\beta}}", shift left, from=1-4, to=1-5]
	\arrow["{\overline{\Gamma}}"{xshift=0.1cm}, shift right, shorten <=6pt, shorten >=6pt, Rightarrow, from=0, to=1]
	\arrow["{\overline{\Delta}}"{xshift=0.1cm}, shift right, shorten <=6pt, shorten >=6pt, Rightarrow, from=3, to=2]
\end{tikzcd}\\
$=$ \\
\begin{tikzcd}[ampersand replacement=\&]
	\bullet \& \bullet \& \bullet \& \bullet \& \bullet \& \bullet
	\arrow[""{name=0, anchor=center, inner sep=0}, "{\overline{\alpha}}", shift left, from=1-2, to=1-3]
	\arrow[""{name=1, anchor=center, inner sep=0}, "{\overline{\alpha'}}"', shift right, curve={height=18pt}, from=1-2, to=1-3]
	\arrow["{\labL(\overline{U})}", from=1-1, to=1-2]
	\arrow["{\labL(\overline{V})}", from=1-3, to=1-4]
	\arrow["{\labL(\overline{W})}", from=1-5, to=1-6]
	\arrow[""{name=2, anchor=center, inner sep=0}, "{\overline{\beta'}}"', shift right, from=1-4, to=1-5]
	\arrow[""{name=3, anchor=center, inner sep=0}, "{\overline{\beta}}", shift left, curve={height=-18pt}, from=1-4, to=1-5]
	\arrow["{\overline{\Gamma}}"{xshift=0.1cm}, shift right, shorten <=6pt, shorten >=6pt, Rightarrow, from=0, to=1]
	\arrow["{\overline{\Delta}}"{xshift=0.1cm}, shift right, shorten <=6pt, shorten >=6pt, Rightarrow, from=3, to=2]
\end{tikzcd}
    \end{center}
    This follows directly from the pasting theorem for 2-categories, since these are two composite of the same pasting diagram in the 2-category $\gk(S,T)$. \black We will show the explicit calculation for the pasting diagram below.
\[\begin{tikzcd}[ampersand replacement=\&]
	S \& A \& B \& C \& D \& T
	\arrow["f"{description}, from=1-1, to=1-2]
	\arrow[""{name=0, anchor=center, inner sep=0}, "{a}", curve={height=-12pt}, from=1-2, to=1-3]
	\arrow[""{name=1, anchor=center, inner sep=0}, "{a'}"', curve={height=12pt}, from=1-2, to=1-3]
	\arrow[""{name=2, anchor=center, inner sep=0}, "g"{description}, from=1-3, to=1-4]
	\arrow[""{name=3, anchor=center, inner sep=0}, "{b}", curve={height=-12pt}, from=1-4, to=1-5]
	\arrow[""{name=4, anchor=center, inner sep=0}, "{b'}"', curve={height=12pt}, from=1-4, to=1-5]
	\arrow["h"{description}, from=1-5, to=1-6]
	\arrow[""{name=5, anchor=center, inner sep=0}, "p", curve={height=-40pt}, from=1-1, to=1-6]
	\arrow[""{name=6, anchor=center, inner sep=0}, "q"', curve={height=40pt}, from=1-1, to=1-6]
	\arrow["\phi", shorten <=8pt, shorten >=8pt, Rightarrow, from=5, to=2]
	\arrow["\psi", shorten <=8pt, shorten >=8pt, Rightarrow, from=2, to=6]
	\arrow["\alpha"{xshift=0.1cm}, shorten <=6pt, shorten >=6pt, Rightarrow, from=0, to=1]
	\arrow["\beta"{xshift=0.1cm}, shorten <=6pt, shorten >=6pt, Rightarrow, from=3, to=4]
\end{tikzcd}\]
 In this case, the wanted equation is true because it becomes the local interchange for the following 3-cells. 
\[\begin{tikzcd}[ampersand replacement=\&]
	p \& hbgaf \&\& {hbga'f} \&\& {hb'ga'f} \& q
	\arrow["\psi", from=1-6, to=1-7]
	\arrow["\phi", from=1-1, to=1-2]
	\arrow[""{name=0, anchor=center, inner sep=0}, "1\circ\alpha\circ1", curve={height=-12pt}, from=1-2, to=1-4]
	\arrow[""{name=1, anchor=center, inner sep=0}, "1\circ\beta\circ1", curve={height=-12pt}, from=1-4, to=1-6]
	\arrow[""{name=2, anchor=center, inner sep=0}, "{1\circ\alpha'\circ1}"', curve={height=12pt}, from=1-2, to=1-4]
	\arrow[""{name=3, anchor=center, inner sep=0}, "{1\circ\beta'\circ1}"', curve={height=12pt}, from=1-4, to=1-6]
	\arrow["1\circ\Gamma\circ1", shift right=4, shorten <=6pt, shorten >=6pt, Rightarrow, from=0, to=2]
	\arrow["1\circ\Delta\circ1", shift right=4, shorten <=6pt, shorten >=6pt, Rightarrow, from=1, to=3] 
\end{tikzcd}\]
\end{proof}

\begin{eg}
    \Cref{thm:easy-gen-loc-int} for the pasting scheme 
    \begin{center}
\begin{tikzcd}[ampersand replacement=\&]
	\bullet \&\& \bullet
	\arrow[""{name=0, anchor=center, inner sep=0}, shift left=1, curve={height=-24pt}, from=1-1, to=1-3]
	\arrow[""{name=1, anchor=center, inner sep=0}, from=1-1, to=1-3]
	\arrow[""{name=2, anchor=center, inner sep=0}, shift right=1, curve={height=24pt}, from=1-1, to=1-3]
	\arrow["F_\alpha"{description}, draw=none, from=0, to=1]
	\arrow["F_\beta"{description}, draw=none, from=1, to=2]
\end{tikzcd}
\end{center}
is exactly the normal local interchange for Gray-categories. 
\end{eg}

\begin{rmk}
\Cref{thm:easy-gen-loc-int} says that given a 2-dimensional pasting diagram and a choice of composite of it, then we can apply two 3-cells (on different faces) in any order we like. Moreover, using \Cref{thm:easy-nat-3-cell}, we can see that this remains true even if we change the composite at any point with the unique invertible 3-cell given by the Gray-categorical pasting theorem \cite[Theorem~4.24]{div:past-thm-gray}. 

More formally, we consider seven different possible composites of a labelling $\labL$ and we write $X,Y,A,B,C,D$ for $X:=\labL\overline{K_X}, Y:=\labL\overline{K_Y}$, etc. When we change the labelling using the various $\labL^{\alpha'},\labL^{\beta'}$ and $\labL^{\alpha',\beta'}$ we will indicate it with sup/sub-scripts. For instance $X^{\alpha'}:=\labL^{\alpha'}(\overline{K_X})$, $Y_{\beta'}:=\labL^{\beta'}(\overline{K_Y})$ and $A^{\alpha'}_{\beta'}:=\labL^{\alpha'}(\overline{K_A})$. From various applications of \Cref{thm:easy-nat-3-cell}, \Cref{thm:easy-gen-loc-int} and the uniqueness of the isomorphism between different composites of a labelling, we can deduce that the following diagram is commutative. 
\[\begin{tikzcd}[ampersand replacement=\&]
	X \& A \& {A^{\alpha'}} \& {B^{\alpha'}} \& {B^{\alpha'}_{\beta'}} \& {Y^{\alpha'}_{\beta'}} \\
	\& C \& {C_{\beta'}} \& {D_{\beta'}} \& {D_{\beta'}^{\alpha'}}
	\arrow["\gamma", from=1-1, to=1-2]
	\arrow["\gamma"', from=1-1, to=2-2]
	\arrow["{\Gamma^\ast}", from=1-2, to=1-3]
	\arrow["\gamma", from=1-3, to=1-4]
	\arrow["{\Delta^\ast}", from=1-4, to=1-5]
	\arrow["{\Delta^\ast}"', from=2-2, to=2-3]
	\arrow["\gamma"', from=2-3, to=2-4]
	\arrow["{\Gamma^\ast}"', from=2-4, to=2-5]
	\arrow["\gamma", from=1-5, to=1-6]
	\arrow["\gamma"', from=2-5, to=1-6]
\end{tikzcd}\]
The commutativity of the diagrams (2), (3), (4) and (7) follows by \Cref{thm:easy-nat-3-cell}, (1) and (5) commute by \cite[Theorem~4.24]{div:past-thm-gray}, and finally (6) by \Cref{thm:easy-gen-loc-int}.
\begin{center}
    \scalebox{0.8}{
\begin{tikzcd}[ampersand replacement=\&]
	\& A \& {A^{\alpha'}} \& {B^{\alpha'}} \& {B^{\alpha'}_{\beta'}} \\
	X \& B \&\&\&\& {Y^{\alpha'}_{\beta'}} \\
	\& D \&\& {D^{\alpha'}} \\
	\& C \& {C_{\beta'}} \& {D_{\beta'}} \& {D_{\beta'}^{\alpha'}}
	\arrow["\gamma", from=2-1, to=1-2]
	\arrow["\gamma"', from=2-1, to=4-2]
	\arrow["{\Gamma^\ast}", from=1-2, to=1-3]
	\arrow["\gamma", from=1-3, to=1-4]
	\arrow["{\Delta^\ast}", from=1-4, to=1-5]
	\arrow["{\Delta^\ast}"', from=4-2, to=4-3]
	\arrow[""{name=0, anchor=center, inner sep=0}, "\gamma"', from=4-3, to=4-4]
	\arrow[""{name=1, anchor=center, inner sep=0}, "{\Gamma^\ast}"', from=4-4, to=4-5]
	\arrow["\gamma", from=1-5, to=2-6]
	\arrow["\gamma"', from=4-5, to=2-6]
	\arrow["\gamma"{description}, from=1-2, to=2-2]
	\arrow[""{name=2, anchor=center, inner sep=0}, "{\Gamma^\ast}"{description}, from=2-2, to=1-4]
	\arrow["{\Delta^\ast}"{description}, from=3-2, to=4-4]
	\arrow[""{name=3, anchor=center, inner sep=0}, "\gamma"{description}, from=3-2, to=4-2]
	\arrow[""{name=4, anchor=center, inner sep=0}, "\gamma"{description}, from=2-2, to=3-2]
	\arrow[""{name=5, anchor=center, inner sep=0}, "\gamma"{description}, from=1-5, to=4-5]
	\arrow[""{name=6, anchor=center, inner sep=0}, "{\Gamma^\ast}"{description}, from=3-2, to=3-4]
	\arrow["{\Delta^\ast}"{description}, from=3-4, to=4-5]
	\arrow[""{name=7, anchor=center, inner sep=0}, "\gamma"{description}, from=1-4, to=3-4]
	\arrow["{(1)}"{description}, draw=none, from=2-1, to=4]
	\arrow["{(3)}"{description}, draw=none, from=2, to=6]
	\arrow["{(4)}"{description}, draw=none, from=7, to=5]
	\arrow["{(5)}"{description}, draw=none, from=5, to=2-6]
	\arrow["{(6)}"{description}, draw=none, from=6, to=1]
	\arrow["{(7)}"{description, pos=0.4}, draw=none, from=3, to=0]
	\arrow["{(2)}"{description}, draw=none, from=1-2, to=2]
\end{tikzcd}
}
\end{center}

\end{rmk}

\section{2-to-1 Local Interchange}

In this section we will consider something that happens quite often in practice: when a 3-cell acts on more than one face of a pasting diagram. For instance, let us consider the 2-dimensional pasting below. 
\[\begin{tikzcd}[ampersand replacement=\&]
	\&\& A \\
	S \& B \&\& T \\
	\&\& C
	\arrow[""{name=0, anchor=center, inner sep=0}, "{p_1}", curve={height=-12pt}, from=2-1, to=1-3]
	\arrow["{p_2}", shift left=1, curve={height=-6pt}, from=1-3, to=2-4]
	\arrow[""{name=1, anchor=center, inner sep=0}, "{q_1}"', curve={height=12pt}, from=2-1, to=3-3]
	\arrow["{g_1}"{description}, from=2-1, to=2-2]
	\arrow["h"{description}, from=2-2, to=3-3]
	\arrow["f"{description}, from=2-2, to=1-3]
	\arrow[""{name=2, anchor=center, inner sep=0}, "{g_2}"{description}, from=2-2, to=2-4]
	\arrow["{q_2}"', curve={height=6pt}, from=3-3, to=2-4]
	\arrow["{\alpha_1}"'{pos=0.6}, shorten <=6pt, shorten >=4pt, Rightarrow, from=0, to=2-2]
	\arrow["{\beta_1}"'{pos=0.4,yshift=-0.1cm}, shorten <=4pt, shorten >=6pt, Rightarrow, from=2-2, to=1]
	\arrow["{\alpha_2}"{xshift=0.1cm}, shorten <=6pt, shorten >=6pt, Rightarrow, from=1-3, to=2]
	\arrow["{\beta_2}"{xshift=0.1cm}, shorten <=6pt, shorten >=6pt, Rightarrow, from=2, to=3-3]
\end{tikzcd}\]
We want to show that any 3-cell $\Pi\colon(\alpha_2\circ g_1)\cdot(p_2\circ\alpha_1)\Rrightarrow\alpha'$, induces a natural map as below. 
\begin{center}
\scalebox{0.9}{
\begin{tikzcd}[ampersand replacement=\&]
	\&\& A \\
	S \& B \&\& T \\
	\&\& C
	\arrow["{p_2}", shift left=1, curve={height=-6pt}, from=1-3, to=2-4]
	\arrow[""{name=0, anchor=center, inner sep=0}, "{q_1}"', curve={height=12pt}, from=2-1, to=3-3]
	\arrow["{g_1}"{description}, from=2-1, to=2-2]
	\arrow["h"{description}, from=2-2, to=3-3]
	\arrow["f"{description}, from=2-2, to=1-3]
	\arrow[""{name=1, anchor=center, inner sep=0}, "{g_2}"{description}, from=2-2, to=2-4]
	\arrow["{q_2}"', curve={height=6pt}, from=3-3, to=2-4]
	\arrow[""{name=2, anchor=center, inner sep=0}, "{p_1}", curve={height=-12pt}, from=2-1, to=1-3]
	\arrow["{\beta_1}"', shorten <=4pt, shorten >=6pt, Rightarrow, from=2-2, to=0]
	\arrow["{\alpha_2}"{xshift=0.1cm}, shorten <=6pt, shorten >=6pt, Rightarrow, from=1-3, to=1]
	\arrow["{\beta_2}"{xshift=0.1cm}, shorten <=6pt, shorten >=6pt, Rightarrow, from=1, to=3-3]
	\arrow["{\alpha_1}"', shorten <=6pt, shorten >=4pt, Rightarrow, from=2, to=2-2]
\end{tikzcd}
    \hspace{0.5cm}$\overset{\Pi^\ast}{\Rrightarrow}$\hspace{0.5cm}
\begin{tikzcd}[ampersand replacement=\&]
	\&\& A \\
	S \& B \&\& T \\
	\&\& C
	\arrow["{p_2}", shift left, curve={height=-6pt}, from=1-3, to=2-4]
	\arrow[""{name=0, anchor=center, inner sep=0}, "{q_1}"', curve={height=12pt}, from=2-1, to=3-3]
	\arrow["{g_1}"{description}, from=2-1, to=2-2]
	\arrow["h"{description}, from=2-2, to=3-3]
	\arrow[""{name=1, anchor=center, inner sep=0}, "{g_2}"{description}, from=2-2, to=2-4]
	\arrow["{q_2}"', curve={height=6pt}, from=3-3, to=2-4]
	\arrow["{p_1}", curve={height=-12pt}, from=2-1, to=1-3]
	\arrow["{\beta_1}"'{pos=0.4}, shorten <=4pt, shorten >=6pt, Rightarrow, from=2-2, to=0]
	\arrow["{\beta_2}"{xshift=0.1cm}, shorten <=6pt, shorten >=6pt, Rightarrow, from=1, to=3-3]
	\arrow["{\alpha'}"{xshift=0.1cm}, shift right=5, shorten <=6pt, shorten >=6pt, Rightarrow, from=1-3, to=1]
\end{tikzcd}
}
\end{center}

Now, we need to formalise when two faces $F,G$ of a pasting scheme $\pschG$, given a labelling, result in \emph{composable 2-cells}. In other words, we want to describe when $F\lhd_\pschG G$, i.e. when $\tau_F$ and $\sigma_G$ share at least an edge (see \cite[Definition~4.8]{div:past-thm-gray}), and there exists at least an element of the form $\overline{H}FG\overline{K}\in\cc_\pschG$. We will see that the hypothesis $F\lhd_\pschG G$ will avoid scenarios where the composition is not uniquely determined, such as the pasting scheme below. 
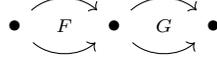
\[\begin{tikzcd}[ampersand replacement=\&]
	\bullet \& \bullet \& \bullet
	\arrow[""{name=0, anchor=center, inner sep=0}, curve={height=-12pt}, from=1-1, to=1-2]
	\arrow[""{name=1, anchor=center, inner sep=0}, curve={height=-12pt}, from=1-2, to=1-3]
	\arrow[""{name=2, anchor=center, inner sep=0}, curve={height=12pt}, from=1-1, to=1-2]
	\arrow[""{name=3, anchor=center, inner sep=0}, curve={height=12pt}, from=1-2, to=1-3]
	\arrow["F"{description}, draw=none, from=0, to=2]
	\arrow["G"{description}, draw=none, from=1, to=3]
\end{tikzcd}\]
We will discuss briefly this special case in \Cref{sec:2-1-special-case}. Before diving in \Cref{lem:F-G-composable}, which will give a complete description of all the possible cases, let us recall some notation. For two faces $F$ and $G$ of a pasting scheme $\pschG$, we write $F\prec_\pschG G$ if and only if there exists a (possibly empty) path from $t_F$ to $s_G$ (see \cite[Definition~4.10]{div:past-thm-gray}). We also denote with ${\lhd_{\pschG}^{t}}$ the transitive closure of ${\lhd_{\pschG}}$. \black

\begin{lemma}\label{lem:F-G-composable}
    Let $\pschG$ be a pasting scheme and $F,G$ two faces. If there exists at least an element of the form $\overline{H}FG\overline{K}\in\cc_\pschG$, then one of the following two conditions holds: 
    \begin{enumerate}[(a)]
        \item $F\prec_\pschG G$ or $G\prec_\pschG F$;
        \item the face obtained by the union of $F$ and $G$ has one of the following shapes 
    \begin{center}
    \scalebox{0.9}
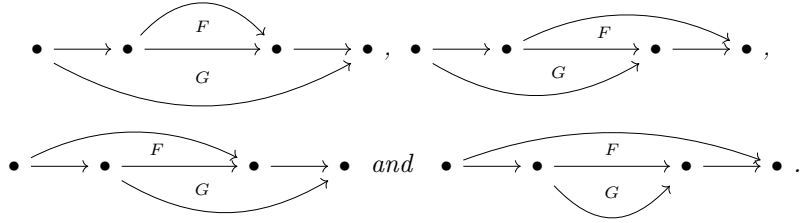
{
\begin{tikzcd}[ampersand replacement=\&]
	\bullet \& \bullet \&\& \bullet \& \bullet
	\arrow[""{name=0, anchor=center, inner sep=0}, from=1-2, to=1-4]
	\arrow[from=1-4, to=1-5]
	\arrow[from=1-1, to=1-2]
	\arrow[""{name=1, anchor=center, inner sep=0}, shift right=1, curve={height=24pt}, from=1-1, to=1-5]
	\arrow[""{name=2, anchor=center, inner sep=0}, curve={height=-24pt}, from=1-2, to=1-4]
	\arrow["G"{description}, draw=none, from=0, to=1]
	\arrow["F"{description}, draw=none, from=2, to=0]
\end{tikzcd}, 
\begin{tikzcd}[ampersand replacement=\&]
	\bullet \& \bullet \&\& \bullet \& \bullet
	\arrow[""{name=0, anchor=center, inner sep=0}, from=1-2, to=1-4]
	\arrow[from=1-1, to=1-2]
	\arrow[""{name=1, anchor=center, inner sep=0}, shift right=1, curve={height=18pt}, from=1-1, to=1-4]
	\arrow[""{name=2, anchor=center, inner sep=0}, curve={height=-18pt}, from=1-2, to=1-5]
	\arrow[from=1-4, to=1-5]
	\arrow["G"{description}, draw=none, from=0, to=1]
	\arrow["F"{description}, draw=none, from=2, to=0]
\end{tikzcd}, }\\
\scalebox{0.9}{
\begin{tikzcd}[ampersand replacement=\&]
	\bullet \& \bullet \&\& \bullet \& \bullet
	\arrow[""{name=0, anchor=center, inner sep=0}, from=1-2, to=1-4]
	\arrow[from=1-1, to=1-2]
	\arrow[from=1-4, to=1-5]
	\arrow[""{name=1, anchor=center, inner sep=0}, shift right=1, curve={height=18pt}, from=1-2, to=1-5]
	\arrow[""{name=2, anchor=center, inner sep=0}, curve={height=-18pt}, from=1-1, to=1-4]
	\arrow["G"{description}, draw=none, from=0, to=1]
	\arrow["F"{description}, draw=none, from=2, to=0]
\end{tikzcd}} and \scalebox{0.9}{
\begin{tikzcd}[ampersand replacement=\&]
	\bullet \& \bullet \&\& \bullet \& \bullet
	\arrow[""{name=0, anchor=center, inner sep=0}, from=1-2, to=1-4]
	\arrow[from=1-1, to=1-2]
	\arrow[from=1-4, to=1-5]
	\arrow[""{name=1, anchor=center, inner sep=0}, shift right=1, curve={height=18pt}, from=1-2, to=1-4]
	\arrow[""{name=2, anchor=center, inner sep=0}, curve={height=-18pt}, from=1-1, to=1-5]
	\arrow["G"{description}, draw=none, from=0, to=1]
	\arrow["F"{description}, draw=none, from=2, to=0]
\end{tikzcd}}.
\end{center}
    \end{enumerate}
\end{lemma}

\begin{proof}
    First, we recall that an object of $\cc_\pschG$ is a strict linear extension of $\lhd_\pschG$ (see \cite[Definition~4.15]{div:past-thm-gray}), so if there exists $\overline{H}FG\overline{K}\in\cc_\pschG$ then there cannot be any face $U\in\pschG_2$ such that $F\lhd^t_\pschG U \lhd^t_\pschG G$. Moreover, since $G$ comes after $F$ in the object $\overline{H}FG\overline{K}$, it is not possible that $G\lhd_\pschG^t F$, hence $G\cancel{\lhd_{\pschG}^{t}}F$. \\

    If also $F\cancel{\lhd_{\pschG}^{t}}G$, then it must follow (a), since in \cite[Proposition~4.14]{div:past-thm-gray} it is proven that (a) is indeed equivalent to the statement
    \begin{center}
        $F\cancel{\lhd_{\pschG}^{t}} G$ and $G\cancel{\lhd_{\pschG}^{t}} F$. \\
    \end{center}

    Now, we have left to assume $F\lhd_\pschG^t G$ and prove that (b) holds. First, we notice that actually $F\lhd_\pschG G$, since $\lhd_\pschG^t$ is the transitive closure of $\lhd_\pschG$ but there cannot be any face $U$ such that $F\lhd^t_\pschG U \lhd^t_\pschG G$. 
    
    Then, $F\lhd_\pschG G$ means that $\tau_F$, the lowermost path of $F$, and $\sigma_G$, the uppermost path of $G$ share at least one edge. 
    This implies that the intersection $\tau_F\cap\sigma_G$ must be a connected path. In fact, if we assume $\tau_F\cap\sigma_G$ not connected, there must exist a face $U\in\pschG_2$ such that $F\lhd^t_\pschG U \lhd^t_\pschG G$ as shown below. 
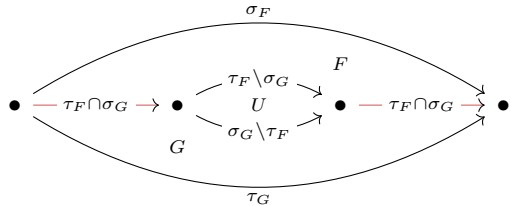
\[\begin{tikzcd}[ampersand replacement=\&]
	\bullet \&\& \bullet \&\& \bullet \&\& \bullet
	\arrow[""{name=0, anchor=center, inner sep=0}, "{\tau_F\cap\sigma_G}"{description}, draw={rgb,255:red,214;green,92;blue,92}, from=1-1, to=1-3]
	\arrow[""{name=1, anchor=center, inner sep=0}, "{\tau_F\setminus\sigma_G}"{description}, curve={height=-12pt}, from=1-3, to=1-5]
	\arrow[""{name=2, anchor=center, inner sep=0}, "{\sigma_G\setminus\tau_F}"{description}, curve={height=12pt}, from=1-3, to=1-5]
	\arrow[""{name=3, anchor=center, inner sep=0}, "{\sigma_F}", curve={height=-40pt}, from=1-1, to=1-7]
	\arrow[""{name=4, anchor=center, inner sep=0}, "{\tau_F\cap\sigma_G}"{description}, draw={rgb,255:red,214;green,92;blue,92}, from=1-5, to=1-7]
	\arrow[""{name=5, anchor=center, inner sep=0}, "{\tau_G}"', curve={height=40pt}, from=1-1, to=1-7]
	\arrow["U"{description}, draw=none, from=1, to=2]
	\arrow["G"{description}, draw=none, from=0, to=5]
	\arrow["F"{description}, draw=none, from=3, to=4]
\end{tikzcd}\]
Now we will show that the following two conditions must hold:
\begin{enumerate}[(i)]
    \item either $s_F\in\tau_F\cap\sigma_G$ or $s_G\in\tau_F\cap\sigma_G$;
    \item either $t_F\in\tau_F\cap\sigma_G$ or $t_G\in\tau_F\cap\sigma_G$.
\end{enumerate}
    The proof for both is analogous, so we will show only (i), i.e. either $s_F$ or $s_G$ must be in $\tau_F\cap\sigma_G$. Let us assume that neither are in the intersection $\tau_F\cap\sigma_G$ and let us consider the paths $p_F$ and $p_G$ connecting $s$, the source of $\pschG$, to $s_F$ and $s_G$. Then, we obtain the following picture (where $\delta_F$ and $\delta_G$ are the paths connecting $s_F$ and $s_G$ to the start of the intersection). 
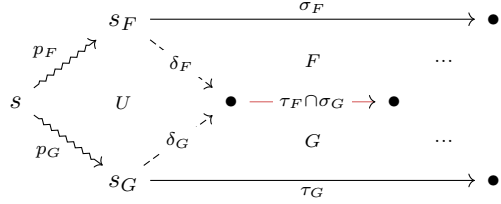
\[\begin{tikzcd}[ampersand replacement=\&]
	\& {s_F} \&\&\&\& \bullet \\
	s \&\& \bullet \&\& \bullet \\
	\& {s_G} \&\&\&\& \bullet
	\arrow[""{name=0, anchor=center, inner sep=0}, "{\tau_F\cap\sigma_G}"{description}, draw={rgb,255:red,214;green,92;blue,92}, from=2-3, to=2-5]
	\arrow["{\delta_F}"{description}, dashed, from=1-2, to=2-3]
	\arrow["{\delta_G}"{description}, dashed, from=3-2, to=2-3]
	\arrow["{p_G}"', squiggly, from=2-1, to=3-2]
	\arrow["U"{description}, draw=none, from=1-2, to=3-2]
	\arrow["{p_F}", squiggly, from=2-1, to=1-2]
	\arrow["{...}"{description}, draw=none, from=2-5, to=1-6]
	\arrow[""{name=1, anchor=center, inner sep=0}, "{\sigma_F}", from=1-2, to=1-6]
	\arrow["{...}"{description}, draw=none, from=2-5, to=3-6]
	\arrow[""{name=2, anchor=center, inner sep=0}, "{\tau_G}"', from=3-2, to=3-6]
	\arrow["F"{description}, draw=none, from=1, to=0]
	\arrow["G"{description}, draw=none, from=0, to=2]
\end{tikzcd}\]
   Then, $p_F,p_G,\delta_F$ and $\delta_G$ form a face $U$ such that $F\lhd_\pschG U\lhd_\pschG G$, which contradicts the hypothesis.

   The four different combinations of (i) and (ii) correspond to the possible pictures in (b). For instance, the first picture correspond to the case $s_F,t_F\in\tau_F\cap\sigma_G$. 
    \end{proof}
\begin{notation*}
    Given two faces $F$ and $G$ of a pasting scheme, if $F\lhd_\pschG G$ and they satisfy the hypothesis of \Cref{lem:F-G-composable}, i.e. (b) holds, we will write $F\lessdot G$.  
\end{notation*}

Starting from a pasting scheme $\pschG$, we will now define a new one that coincides with $\pschG$ everywhere except that it \emph{glues} together $F$ and $G$. 
\begin{notation*}
\begin{itemize}
\item[] 
    \item   For a pasting scheme $\pschG$ we will denote with $\pschG_0,\pschG_1,\pschG_2$ the sets of vertices, edges and faces, respectively. 

    \item For a pasting scheme $\pschG$, we write with $int(\pschG_0):=\pschG_0\setminus\lbrace s,t\rbrace$ where we recall that $s$ and $t$ are the source and sink (i.e. target) of the pasting scheme. 

    \item Since each face of a pasting scheme can be seen itself as a pasting scheme, we will use the same notation as above for faces too. For instance, let $\pschG$ be a pasting scheme and $F\in\pschG_2$, then we will write $F_0$ for the set of vertices surrounding $F$ and $int(F_0):=F_0\setminus\lbrace s_F,t_F\rbrace$. 
\end{itemize}
 
\end{notation*}

\begin{defn}
    Given a pasting scheme $\pschG$ and $F\lessdot G\in\pschG_2$, we define the pasting scheme $\pschG^{FG}$ as:
    \begin{itemize}
        \item $(\pschG^{FG})_0:=\pschG_0\setminus(int(F_0)\cap int(G_0))$;
        \item $(\pschG^{FG})_1:=\pschG_1\setminus(F_1\cap G_1)$;
        \item $(\pschG^{FG})_2:=(\pschG_2\setminus\lbrace F,G\rbrace)\amalg\lbrace F\#G\rbrace$ where the vertices and edges of $F\# G$ \black are defined as
        \begin{center}
            $(F\#G)_0:=(F_0\cup G_0)\setminus(int(F_0)\cap int(G_0))$ and \vspace{0.1cm}\\
            $(F\#G)_1:=F_1\Delta G_1 =(F_1\cup G_1)\setminus (F_1\cap G_1)$ \hspace{0.5cm}
        \end{center} (we recall that $\Delta$ usually denotes the symmetric difference of two sets). 
    \end{itemize}
\end{defn}
\begin{eg}
    Below we give examples of pasting schemes $\pschG$ and the resulting $\pschG^{FG}$.
    \begin{center}
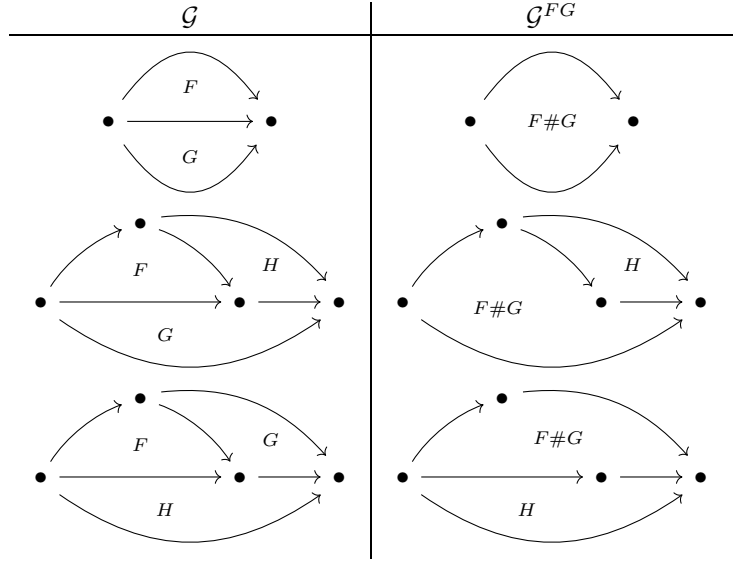
      \begin{tabular}{c|c}
         $\pschG$ & $\pschG^{FG}$ \\
         \hline
\begin{tikzcd}[ampersand replacement=\&]
	\bullet \&\& \bullet
	\arrow[""{name=0, anchor=center, inner sep=0}, shift left=1, curve={height=-24pt}, from=1-1, to=1-3]
	\arrow[""{name=1, anchor=center, inner sep=0}, from=1-1, to=1-3]
	\arrow[""{name=2, anchor=center, inner sep=0}, shift right=1, curve={height=24pt}, from=1-1, to=1-3]
	\arrow["F"{description}, draw=none, from=0, to=1]
	\arrow["G"{description}, draw=none, from=1, to=2]
\end{tikzcd} & 
\begin{tikzcd}[ampersand replacement=\&]
	\bullet \&\& \bullet
	\arrow[""{name=0, anchor=center, inner sep=0}, shift left=1, curve={height=-24pt}, from=1-1, to=1-3]
	\arrow[""{name=1, anchor=center, inner sep=0}, shift right=1, curve={height=24pt}, from=1-1, to=1-3]
	\arrow["{F\#G}"{description}, draw=none, from=0, to=1]
\end{tikzcd}\\
\begin{tikzcd}[ampersand replacement=\&]
	\& \bullet \\
	\bullet \&\& \bullet \& \bullet
	\arrow[""{name=0, anchor=center, inner sep=0}, from=2-1, to=2-3]
	\arrow[""{name=1, anchor=center, inner sep=0}, from=2-3, to=2-4]
	\arrow[""{name=2, anchor=center, inner sep=0}, shift right=1, curve={height=24pt}, from=2-1, to=2-4]
	\arrow[curve={height=-6pt}, from=2-1, to=1-2]
	\arrow[curve={height=-6pt}, from=1-2, to=2-3]
	\arrow[""{name=3, anchor=center, inner sep=0}, shift left=1, curve={height=-12pt}, from=1-2, to=2-4]
	\arrow["F"{description}, draw=none, from=1-2, to=0]
	\arrow["H"{description}, draw=none, from=3, to=1]
	\arrow["G"{description}, draw=none, from=0, to=2]
\end{tikzcd} & 
\begin{tikzcd}[ampersand replacement=\&]
	\& \bullet \\
	\bullet \&\& \bullet \& \bullet
	\arrow[""{name=0, anchor=center, inner sep=0}, from=2-3, to=2-4]
	\arrow[""{name=1, anchor=center, inner sep=0}, shift right=1, curve={height=24pt}, from=2-1, to=2-4]
	\arrow[""{name=2, anchor=center, inner sep=0}, curve={height=-6pt}, from=2-1, to=1-2]
	\arrow[curve={height=-6pt}, from=1-2, to=2-3]
	\arrow[""{name=3, anchor=center, inner sep=0}, shift left=1, curve={height=-12pt}, from=1-2, to=2-4]
	\arrow["H"{description}, draw=none, from=3, to=0]
	\arrow["{F\#G}"{description}, draw=none, from=2, to=1]
\end{tikzcd}\\
\begin{tikzcd}[ampersand replacement=\&]
	\& \bullet \\
	\bullet \&\& \bullet \& \bullet
	\arrow[""{name=0, anchor=center, inner sep=0}, from=2-1, to=2-3]
	\arrow[""{name=1, anchor=center, inner sep=0}, from=2-3, to=2-4]
	\arrow[""{name=2, anchor=center, inner sep=0}, shift right=1, curve={height=24pt}, from=2-1, to=2-4]
	\arrow[curve={height=-6pt}, from=2-1, to=1-2]
	\arrow[curve={height=-6pt}, from=1-2, to=2-3]
	\arrow[""{name=3, anchor=center, inner sep=0}, shift left=1, curve={height=-12pt}, from=1-2, to=2-4]
	\arrow["F"{description}, draw=none, from=1-2, to=0]
	\arrow["G"{description}, draw=none, from=3, to=1]
	\arrow["H"{description}, draw=none, from=0, to=2]
\end{tikzcd}   & 
\begin{tikzcd}[ampersand replacement=\&]
	\& \bullet \\
	\bullet \&\& \bullet \& \bullet
	\arrow[""{name=0, anchor=center, inner sep=0}, from=2-1, to=2-3]
	\arrow[from=2-3, to=2-4]
	\arrow[""{name=1, anchor=center, inner sep=0}, shift right=1, curve={height=24pt}, from=2-1, to=2-4]
	\arrow[curve={height=-6pt}, from=2-1, to=1-2]
	\arrow[""{name=2, anchor=center, inner sep=0}, shift left=1, curve={height=-12pt}, from=1-2, to=2-4]
	\arrow["H"{description}, draw=none, from=0, to=1]
	\arrow["{F\#G}"{description}, draw=none, from=2, to=0]
\end{tikzcd}
    \end{tabular}  
    \end{center} 
    \end{eg}

\begin{prop}
    There is an equivalence $S_{FG}\colon\cc_{\pschG^{FG}}\to\cc_\pschG$.
\end{prop}

\begin{proof}
    Since any functor between contractible groupoids is an equivalence, it is enough to define a functor $S_{FG}$. 
    
   For any object $\overline{U}(F\#G)\overline{V}\in\cc_{\pschG^{FG}}$ we define $S_{FG}(\overline{U}(F\#G)\overline{V}):=\overline{U}FG\overline{V}$. Regarding maps, it suffices to define the functor on generating morphisms. Using notation as in \cite[Definitions~4.5]{div:past-thm-gray}, let us consider the generating morphism $.\widehat{HK}.:=\overline{U}\widehat{HK}\overline{V}\colon\overline{U}HK\overline{V}\to\overline{U}KH\overline{V}\in\cc_{\pschG^{FG}}$. \black 
   \begin{itemize}
       \item If $F\#G\notin\lbrace H,K\rbrace$, then $.\widehat{HK}.$ is sent to the same morphism whiskered by $FG$ instead of $F\#G$. 
       \item If $F\#G\in\lbrace H,K\rbrace$, we define 
       \begin{center}
     $S_{FG}(.\widehat{(F\#G)K}.):=.\widehat{FK}.\circ.\widehat{GK}.$ and \\
     $S_{FG}(.\widehat{H(F\#G)}.):=.\widehat{HG}.\circ.\widehat{HF}.$  
       \end{center} 
   \end{itemize} 
\end{proof}



We underline that, since the categories $\cc_\pschG$ and $\cc_{\pschG^{FG}}$ have a finite number of objects, there always exists a pseudo-inverse of $S_{FG}$. 
Moreover, since both $\cc_\pschG$ and $\cc_{\pschG^{FG}}$ are contractible groupoids, we can describe the possible pseudo-inverses more explicitly. 

\begin{rmk}
   Fixing a pseudo-inverse $J_{FG}\colon\cc_\pschG\to\cc_{\pschG^{FG}}$ of $S_{FG}$ consists of choosing for any object $\overline{H}\in\cc_\pschG$ an object $\overline{H}_{FG}\in\cc_\pschG$ of the form $\overline{H}_{FG}=\overline{U}FG\overline{V}$, $\overline{H}_{FG}:=S_{FG}J_{FG}(\overline{H})$. Then, functoriality follows directly by the contractibility of the codomain, since for any pair of objects there exists a unique map between them.  
\end{rmk}

\begin{prop}\label{prop:re-labelling-F+G}
    For any pasting scheme $\pschG$ and any labelling $\labL$ in a Gray-category $\gk$, there exists a labelling $\labL_{FG}$ of $\pschG^{FG}$ such that for any pseudo-inverse $J_{FG}=J\colon\cc_\pschG\to\cc_{\pschG^{FG}}$ of $S_{FG}$ there is a unique natural isomorphism as below.  
    \begin{center}
\begin{tikzcd}[ampersand replacement=\&]
	{\cc_{\pschG}} \&\& {\gk(S,T)[p,q]} \\
	{\cc_{\pschG^{FG}}}
	\arrow["\labL", from=1-1, to=1-3]
	\arrow["J"', from=1-1, to=2-1]
	\arrow[""{name=0, anchor=center, inner sep=0}, "{\labL_{FG}}"', from=2-1, to=1-3]
	\arrow["\cong"{description}, draw=none, from=1-1, to=0]
\end{tikzcd}
    \end{center}
    Moreover, for any $\alpha=\labL(H_\alpha)$ with $H_\alpha\in\pschG_2\setminus\lbrace F,G\rbrace$ and $\alpha'$ a 2-cell parallel to $\alpha$, then $(\labL_{FG})^{\alpha'}={(\labL^{\alpha'})}_{FG}$. 
\end{prop}

\begin{proof}
    We define $\labL_{FG}:=\labL S_{FG}$, which clearly provides the isomorphism $\labL\cong\labL_{FG}J$ for any $J$. The uniqueness follows directly by \cite{div:past-thm-gray} since each component of the natural isomorphism must be the unique comparison between two different compositions of the pasting diagram associated with $\labL$. 

    The second part, i.e.\ the equality $(\labL_{FG})^{\alpha'}={(\labL^{\alpha'})}_{FG}$, is true because $(-)^{\alpha'}$ acts only on the image of $H_\alpha$ and $(-)_{FG}$ on $\lbrace F,G\rbrace$ which do not intersect by hypothesis. 
\end{proof}

\begin{notation*}
    We will denote with $\labL^{\alpha'}_{FG}$ the labelling $(\labL_{FG})^{\alpha'}={(\labL^{\alpha'})}_{FG}$. 
\end{notation*}

The next Lemma says that the unique isomorphism given by $\labL_{FG}$ interacts well with any natural transformation $\Gamma^\ast$ induced by a 3-cell $\Gamma\colon\alpha\Rrightarrow\alpha'$ acting on a 2-cell $\alpha=\labL(H_\alpha)$ relative to a face of $H_\alpha\in\pschG_2\setminus\lbrace F,G\rbrace$.

\begin{lemma}\label{lem:alpha'-and-FG}
    Let $\pschG$ be a pasting scheme, $F,G,H\in\pschG_2$ different faces, $\labL$ a labelling of $\pschG$ in a Gray-category $\gk$ and $\labL(H)=\alpha$. For any pseudo-inverse $J$ of $S_{FG}$ and any 3-cell $\Gamma\colon\alpha\Rrightarrow\alpha'$ the the following equation holds.
    \begin{center}
\begin{tikzcd}[ampersand replacement=\&]
	{\cc_{\pschG}} \&\& {\gk(S,T)[p,q]} \\
	{\cc_{\pschG^{FG}}}
	\arrow[""{name=0, anchor=center, inner sep=0}, "\labL", curve={height=-12pt}, from=1-1, to=1-3]
	\arrow[""{name=1, anchor=center, inner sep=0}, "J"', from=1-1, to=2-1]
	\arrow[""{name=2, anchor=center, inner sep=0}, "{\labL_{FG}^{\alpha'}}"', curve={height=18pt}, from=2-1, to=1-3]
	\arrow[""{name=3, anchor=center, inner sep=0}, "{\labL^{\alpha'}}"{description}, curve={height=12pt}, from=1-1, to=1-3]
	\arrow["{\Gamma^\ast}", shorten <=6pt, shorten >=6pt, Rightarrow, from=0, to=3]
	\arrow["\cong"{description}, draw=none, from=2, to=1]
\end{tikzcd} $=$
\begin{tikzcd}[ampersand replacement=\&]
	{\cc_{\pschG}} \&\& {\gk(S,T)[p,q]} \\
	{\cc_{\pschG^{FG}}}
	\arrow[""{name=0, anchor=center, inner sep=0}, "\labL", curve={height=-12pt}, from=1-1, to=1-3]
	\arrow["J"', from=1-1, to=2-1]
	\arrow[""{name=1, anchor=center, inner sep=0}, "{\labL_{FG}}"{description}, curve={height=-6pt}, from=2-1, to=1-3]
	\arrow[""{name=2, anchor=center, inner sep=0}, "{\labL_{FG}^{\alpha'}}"', curve={height=18pt}, from=2-1, to=1-3]
	\arrow["{\Gamma^\ast}"{yshift=-0.1cm,xshift=0.1cm}, shorten <=9pt, shorten >=4pt, Rightarrow, from=1, to=2]
	\arrow["\cong"{description}, draw=none, from=1, to=0]
\end{tikzcd}
    \end{center}
\end{lemma}

\begin{proof}
    This follows directly by naturality of $\Gamma^\ast$ (\Cref{thm:easy-nat-3-cell}). Indeed, the left hand side of the required equation becomes the composite going left and then bottom of the commutative square below, and the right hand side the other one. 
\[\begin{tikzcd}[ampersand replacement=\&]
	{X:=\labL(\overline{K})} \& {\labL^{\alpha'}(\overline{K})=:X^{\alpha'}} \\
	{X_{FG}:=\labL_{FG}J(\overline{K})} \& {\labL^{\alpha'}_{FG}J(\overline{K})=:X^{\alpha'}_{FG}}
	\arrow["{\Gamma^\ast_{\overline{K}}}", from=1-1, to=1-2]
	\arrow["\gamma"', from=1-1, to=2-1]
	\arrow["{\Gamma^\ast_{J\overline{K}}}"', from=2-1, to=2-2]
	\arrow["\gamma", from=1-2, to=2-2]
\end{tikzcd}\]
\end{proof}

Let $\pschG$ be a pasting scheme, $F\lessdot G$ faces of $\pschG$ and $\labL$ a labelling in a Gray-category $\gk$ with $\alpha_F=\labL(F)$ and $\alpha_G=\labL(G)$. We denote with $\alpha_G.\alpha_F$ the \emph{pasting} of $\alpha_G$ and $\alpha_F$. The table below gives the precise formulas for $\alpha_G.\alpha_F$ in each of the four possible cases.
\begin{center}
\begin{tabular}{c|c}
     $F\#G$ & $\alpha_G.\alpha_F$  \\ \hline
     \scalebox{0.9}{
\begin{tikzcd}[ampersand replacement=\&]
	A \& B \& C \& D
	\arrow[""{name=0, anchor=center, inner sep=0}, "r"{description}, from=1-2, to=1-3]
	\arrow["i", from=1-1, to=1-2]
	\arrow[""{name=1, anchor=center, inner sep=0}, "q"', shift right, curve={height=18pt}, from=1-1, to=1-4]
	\arrow[""{name=2, anchor=center, inner sep=0}, "p", shift left, curve={height=-18pt}, from=1-2, to=1-3]
	\arrow["j", from=1-3, to=1-4]
	\arrow["{\alpha_G}"{xshift=0.1cm}, shift right=2, shorten <=5pt, shorten >=5pt, Rightarrow, from=0, to=1]
	\arrow["{\alpha_F}"{xshift=0.05cm}, shift right=2, shorten <=7pt, shorten >=7pt, Rightarrow, from=2, to=0]
\end{tikzcd}
     }& $\alpha_G\cdot(j\circ\alpha_F\circ i)$\\
     \scalebox{0.9}{
\begin{tikzcd}[ampersand replacement=\&]
	A \& B \& C \& D
	\arrow["r"{description}, from=1-2, to=1-3]
	\arrow["i", from=1-1, to=1-2]
	\arrow["j"', from=1-3, to=1-4]
	\arrow[""{name=0, anchor=center, inner sep=0}, "p", shift left, curve={height=-18pt}, from=1-2, to=1-4]
	\arrow[""{name=1, anchor=center, inner sep=0}, "q"', shift right, curve={height=18pt}, from=1-1, to=1-3]
	\arrow["{\alpha_F}"{pos=0.7,xshift=0.1cm}, shorten <=6pt, shorten >=1pt, Rightarrow, from=0, to=1-3]
	\arrow["{\alpha_G}"{pos=0.3,xshift=0.1cm}, shorten <=1pt, shorten >=6pt, Rightarrow, from=1-2, to=1]
\end{tikzcd}
     }& $(j\circ\alpha_G)\cdot(\alpha_F\circ i)$\\
     \scalebox{0.9}{
\begin{tikzcd}[ampersand replacement=\&]
	A \& B \& C \& D
	\arrow["r"{description}, from=1-2, to=1-3]
	\arrow["i"', from=1-1, to=1-2]
	\arrow["j", from=1-3, to=1-4]
	\arrow[""{name=0, anchor=center, inner sep=0}, "p", shift left, curve={height=-18pt}, from=1-1, to=1-3]
	\arrow[""{name=1, anchor=center, inner sep=0}, "q"', shift right, curve={height=18pt}, from=1-2, to=1-4]
	\arrow["{\alpha_F}"{pos=0.7,xshift=0.1cm}, shorten <=6pt, shorten >=1pt, Rightarrow, from=0, to=1-2]
	\arrow["{\alpha_G}"{pos=0.3,xshift=0.1cm}, shorten <=1pt, shorten >=6pt, Rightarrow, from=1-3, to=1]
\end{tikzcd}
     }& $(\alpha_G\circ i)\cdot(j\circ \alpha_F)$\\
     \scalebox{0.9}{
\begin{tikzcd}[ampersand replacement=\&]
	A \& B \& C \& D
	\arrow[""{name=0, anchor=center, inner sep=0}, "r"{description}, from=1-2, to=1-3]
	\arrow["i"', from=1-1, to=1-2]
	\arrow["j"', from=1-3, to=1-4]
	\arrow[""{name=1, anchor=center, inner sep=0}, "p", shift left, curve={height=-18pt}, from=1-1, to=1-4]
	\arrow[""{name=2, anchor=center, inner sep=0}, "q"', shift right, curve={height=18pt}, from=1-2, to=1-3]
	\arrow["{\alpha_F}"{xshift=0.1cm}, shift right=2, shorten <=5pt, shorten >=5pt, Rightarrow, from=1, to=0]
	\arrow["{\alpha_G}"{xshift=0.05cm}, shift right=2, shorten <=7pt, shorten >=7pt, Rightarrow, from=0, to=2]
\end{tikzcd}
     }&  $(j\circ \alpha_G\circ i)\cdot \alpha_F$
\end{tabular}
\end{center}

\begin{theorem}\label{thm:hard-nat-3-cell}
    Given any pseudo-inverse $J$ of $S_{FG}$, any 3-cell $\Pi\colon\alpha_G.\alpha_F\Rrightarrow\alpha'$ induces a natural transformation as below. 
\begin{center}
\begin{tikzcd}[ampersand replacement=\&]
	{\cc_{\pschG}} \& {} \& {\gk(S,T)[p,q]} \\
	{\cc_{\pschG^{FG}}} \& {}
	\arrow["\labL", from=1-1, to=1-3]
	\arrow["{\labL^{\alpha'}_{FG}}"', curve={height=12pt}, from=2-1, to=1-3]
	\arrow["J"', from=1-1, to=2-1]
	\arrow["{\overline{\Pi^\ast}}"{yshift=0.1cm,xshift=0.1cm}, shift right=4, shorten <=4pt, shorten >=8pt, Rightarrow, from=1-2, to=2-2]
\end{tikzcd}
\end{center}
\end{theorem}

\begin{proof}
        This follows by putting together \Cref{thm:easy-nat-3-cell}, for the labelling $\labL_{FG}$, and \Cref{prop:re-labelling-F+G} as shown below.  
        \begin{center}
\begin{tikzcd}[ampersand replacement=\&]
	{\cc_{\pschG}} \& {} \& {\gk(S,T)[p,q]} \\
	{\cc_{\pschG^{FG}}} \& {}
	\arrow["\labL", from=1-1, to=1-3]
	\arrow["{\labL^{\alpha'}_{FG}}"', curve={height=12pt}, from=2-1, to=1-3]
	\arrow["J"', from=1-1, to=2-1]
	\arrow["{\overline{\Pi^\ast}}"{yshift=0.1cm,xshift=0.1cm}, shift right=4, shorten <=4pt, shorten >=8pt, Rightarrow, from=1-2, to=2-2]
\end{tikzcd} $:=$
\begin{tikzcd}[ampersand replacement=\&]
	{\cc_{\pschG}} \&\& {\gk(S,T)[p,q]} \\
	{\cc_{\pschG^{FG}}}
	\arrow[""{name=0, anchor=center, inner sep=0}, "\labL", from=1-1, to=1-3]
	\arrow[""{name=1, anchor=center, inner sep=0}, "{\labL^{\alpha'}_{FG}}"', curve={height=12pt}, from=2-1, to=1-3]
	\arrow[""{name=2, anchor=center, inner sep=0}, "J"', from=1-1, to=2-1]
	\arrow[""{name=3, anchor=center, inner sep=0}, "{\labL_{FG}}"{description, pos=0.3}, curve={height=-8pt}, from=2-1, to=1-3]
	\arrow["{{\Pi^\ast}}"{pos=0.7}, shorten <=5pt, shorten >=3pt, Rightarrow, from=3, to=1]
	\arrow["\cong"{description}, draw=none, from=2, to=0]
\end{tikzcd}
        \end{center}
\end{proof}

\begin{rmk}\label{rmk:lab-and-FG-HK}
Let $\pschG$ be a pasting scheme and $F\lessdot G, H\lessdot K$ two pairs of faces.
\begin{enumerate}[(i)]
    \item By definitions, $(\pschG^{FG})^{HK}=(\pschG^{HK})^{FG}=:\pschG^{FG,HK}$. 
    \item The diagram below left is commutative, hence the diagram below right is commutative up to unique isomorphism (for any choice of inverses $J$).  

\begin{center}
\begin{tikzcd}[ampersand replacement=\&]
	{\cc_{(\pschG^{HK})^{FG}}=\cc_{(\pschG^{FG})^{HK}}} \& {\cc_{\pschG^{FG}}} \\
	{\cc_{\pschG^{HK}}} \& {\cc_\pschG}
	\arrow["{S_{FG}}", from=1-2, to=2-2]
	\arrow["{S_{HK}}"', from=2-1, to=2-2]
	\arrow["{S_{FG}}"', from=1-1, to=2-1]
	\arrow["{S_{HK}}", from=1-1, to=1-2]
\end{tikzcd}
\hspace{1cm}
\begin{tikzcd}[ampersand replacement=\&]
	{\cc_\pschG} \& {\cc_{\pschG^{FG}}} \\
	{\cc_{\pschG^{HK}}} \& {\cc_{\pschG^{FG,HK}}}
	\arrow["J"', from=1-1, to=2-1]
	\arrow[""{name=0, anchor=center, inner sep=0}, "J", from=1-1, to=1-2]
	\arrow[""{name=1, anchor=center, inner sep=0}, "J"', from=2-1, to=2-2]
	\arrow["J", from=1-2, to=2-2]
	\arrow["\cong"{description}, draw=none, from=0, to=1]
\end{tikzcd}
\end{center}

    \item The uniqueness of the isomorphism in the Gray-categorical pasting theorem proves that the following equality holds, where the isomorphisms in the diagrams are the unique possible ones, given choices of pseudo-inverses of each $S_{FG}, S^{\pschG^{FG}}_{HK}$ etc.
    \begin{center}
\scalebox{0.9}{\hspace{-1.5cm}
\begin{tikzcd}[ampersand replacement=\&]
	{\cc_{\pschG}} \& {} \\
	{\cc_{\pschG^{FG}}} \& {} \& {\gk(S,T)[p,q]} \\
	{\cc_{\pschG^{FG,HK}}} \& {}
	\arrow["\labL", curve={height=-12pt}, from=1-1, to=2-3]
	\arrow["{\labL_{FG}}"{description}, from=2-1, to=2-3]
	\arrow["{J_{FG}}"', from=1-1, to=2-1]
	\arrow["{{(J^{\pschG^{FG}})}_{HK}}"', from=2-1, to=3-1]
	\arrow["{\labL_{FG,HK}}"', curve={height=12pt}, from=3-1, to=2-3]
	\arrow["\cong"{description}, draw=none, from=1-2, to=2-1]
	\arrow["\cong"{description}, draw=none, from=2-2, to=3-1]
\end{tikzcd} $=$
\begin{tikzcd}[ampersand replacement=\&]
	\& {\cc_{\pschG}} \& {} \\
	{\cc_{\pschG^{FG}}} \& {\cc_{\pschG^{HK}}} \& {} \& {\gk(S,T)[p,q]} \\
	\& {\cc_{\pschG^{FG,HK}}} \& {}
	\arrow["\labL", curve={height=-12pt}, from=1-2, to=2-4]
	\arrow["{J_{FG}}"', from=1-2, to=2-1]
	\arrow["{{(J^{\pschG^{FG}})}_{HK}}"', from=2-1, to=3-2]
	\arrow["{\labL_{FG,HK}}"', curve={height=12pt}, from=3-2, to=2-4]
	\arrow["{\labL_{HK}}"{description}, from=2-2, to=2-4]
	\arrow["{J_{HK}}"{description}, from=1-2, to=2-2]
	\arrow["{{(J^{\pschG^{HK}})}_{FG}}"{description,xshift=0.1cm}, from=2-2, to=3-2]
	\arrow["\cong"{description}, draw=none, from=2-1, to=2-2]
	\arrow["\cong"{description}, draw=none, from=2-3, to=3-2]
	\arrow["\cong"{description}, draw=none, from=1-3, to=2-2]
\end{tikzcd}}
\end{center}

\end{enumerate}

\end{rmk}


\begin{theorem}\label{thm:hard-gen-loc-int}
    Let $\pschG$ be a pasting scheme, $F\lessdot G, H\lessdot K\in\pschG_2$ and $\labL$ a labelling of $\pschG$ in $\gk$ with $\alpha_F,\alpha_G,\beta_H,\beta_K$ the 2-cells in the labelling corresponding to the faces $F,G,H,K$ respectively. Given any two 3-cells $\Pi\colon\alpha_G.\alpha_F\Rrightarrow\alpha'$ and $\Omega\colon\beta_K.\beta_H\to\beta'$, then for any inverses $J$'s of all the $I$'s the following equality holds. 
\begin{center}
\scalebox{0.9}{
\begin{tikzcd}[ampersand replacement=\&]
	{\cc_{\pschG}} \& {} \\
	{\cc_{\pschG^{FG}}} \& {} \& {\gk(S,T)[p,q]} \\
	{\cc_{\pschG^{FG,HK}}} \& {}
	\arrow["\labL", curve={height=-12pt}, from=1-1, to=2-3]
	\arrow["{\labL^{\alpha'}_{FG}}"{description}, from=2-1, to=2-3]
	\arrow["{J}"', from=1-1, to=2-1]
	\arrow["{\overline{\Pi^\ast}}"{xshift=0.1cm}, shorten <=8pt, shorten >=10pt, Rightarrow, from=1-2, to=2-2]
	\arrow["{J}"', from=2-1, to=3-1]
	\arrow["{\labL^{\alpha',\beta'}_{FG,HK}}"', curve={height=12pt}, from=3-1, to=2-3]
	\arrow["{\overline{\Omega^\ast}}"{yshift=0.1cm,xshift=0.1cm}, shorten <=8pt, shorten >=10pt, Rightarrow, from=2-2, to=3-2]
\end{tikzcd} $=$
\begin{tikzcd}[ampersand replacement=\&]
	\& {\cc_{\pschG}} \& {} \\
	{\cc_{\pschG^{FG}}} \& {\cc_{\pschG^{HK}}} \& {} \& {\gk(S,T)[p,q]} \\
	\& {\cc_{\pschG^{FG,HK}}} \& {}
	\arrow["\labL", curve={height=-12pt}, from=1-2, to=2-4]
	\arrow["{J}"', from=1-2, to=2-1]
	\arrow["{\overline{\Omega^\ast}}"{xshift=0.1cm}, shorten <=8pt, shorten >=10pt, Rightarrow, from=1-3, to=2-3]
	\arrow["{J}"', from=2-1, to=3-2]
	\arrow["{\labL^{\alpha',\beta'}_{FG,HK}}"', curve={height=12pt}, from=3-2, to=2-4]
	\arrow["{\overline{\Pi^\ast}}"{yshift=0.1cm,xshift=0.1cm}, shorten <=8pt, shorten >=10pt, Rightarrow, from=2-3, to=3-3]
	\arrow["{\labL^{\beta'}_{HK}}"{description}, from=2-2, to=2-4]
	\arrow["{J}"{description}, from=1-2, to=2-2]
	\arrow["{J}"{description}, from=2-2, to=3-2]
	\arrow["\cong"{description}, draw=none, from=2-1, to=2-2]
\end{tikzcd}}
\end{center}
\end{theorem}
\black 

\begin{proof}
We will show that this follows by applying \Cref{lem:alpha'-and-FG} twice, part (iii) of \Cref{rmk:lab-and-FG-HK} once and finally \Cref{thm:easy-gen-loc-int}. We start writing explicitly the left hand side of the wanted equation, using the definitions of $\overline{\Pi^\ast}$ and $\overline{\Omega^\ast}$ (the unlabelled isomorphisms are the unique possible ones described in \Cref{prop:re-labelling-F+G} and \Cref{rmk:lab-and-FG-HK}). We recall that the pasting diagrams below are in the 2-category of categories, functors and natural transformations. 
\begin{center}
    \scalebox{0.9}{
\begin{tikzcd}[ampersand replacement=\&]
	{\cc_{\pschG}} \\
	{\cc_{\pschG^{FG}}} \&\&\& {\gk(S,T)[p,q]} \\
	{\cc_{\pschG^{FG,HK}}}
	\arrow["\labL", curve={height=-12pt}, from=1-1, to=2-4]
	\arrow[""{name=0, anchor=center, inner sep=0}, "{\labL^{\alpha'}_{FG}}"{description, pos=0.4}, curve={height=6pt}, from=2-1, to=2-4]
	\arrow["{J_{FG}}"', from=1-1, to=2-1]
	\arrow["{J_{FG}^{HK}}"', from=2-1, to=3-1]
	\arrow[""{name=1, anchor=center, inner sep=0}, "{\labL^{\alpha',\beta'}_{FG,HK}}"', curve={height=24pt}, from=3-1, to=2-4]
	\arrow[""{name=2, anchor=center, inner sep=0}, "{\labL^{\alpha'}_{FG,HK}}"{description, pos=0.3}, curve={height=2pt}, from=3-1, to=2-4]
	\arrow[""{name=3, anchor=center, inner sep=0}, "{\labL_{FG}}"{description, pos=0.3}, curve={height=-15pt}, from=2-1, to=2-4]
	\arrow["{{\Omega^\ast}}"{pos=0.6,yshift=-0.1cm,xshift=0.05cm}, shift left=3, shorten <=4pt, shorten >=4pt, Rightarrow, from=2, to=1]
	\arrow["{{\Pi^\ast}}"{xshift=0.1cm},shift left=3, shorten <=5pt, shorten >=5pt, Rightarrow, from=3, to=0]
	\arrow["\cong"{description}, draw=none, from=3-1, to=0]
	\arrow["\cong"{description}, draw=none, from=3, to=1-1]
\end{tikzcd}
$=$ 
\begin{tikzcd}[ampersand replacement=\&]
	{\cc_{\pschG}} \\
	{\cc_{\pschG^{FG}}} \&\&\& {\gk(S,T)[p,q]} \\
	{\cc_{\pschG^{FG,HK}}}
	\arrow["\labL", curve={height=-12pt}, from=1-1, to=2-4]
	\arrow["{J_{FG}}"', from=1-1, to=2-1]
	\arrow["{J_{FG}^{HK}}"', from=2-1, to=3-1]
	\arrow[""{name=0, anchor=center, inner sep=0}, "{\labL^{\alpha',\beta'}_{FG,HK}}"', curve={height=24pt}, from=3-1, to=2-4]
	\arrow[""{name=1, anchor=center, inner sep=0}, "{\labL^{\alpha'}_{FG,HK}}"{description, pos=0.3}, curve={height=6pt}, from=3-1, to=2-4]
	\arrow[""{name=2, anchor=center, inner sep=0}, "{\labL_{FG}}"{description, pos=0.4}, curve={height=-12pt}, from=2-1, to=2-4]
	\arrow[""{name=3, anchor=center, inner sep=0}, "{\labL_{FG,HK}}"{description, pos=0.3}, curve={height=-12pt}, from=3-1, to=2-4]
	\arrow["{{\Omega^\ast}}"{pos=0.6,yshift=-0.1cm,xshift=0.05cm}, shift left=3, shorten <=4pt, shorten >=4pt, Rightarrow, from=1, to=0]
	\arrow["\cong"{description}, draw=none, from=2, to=1-1]
	\arrow["{\Pi^\ast}"{pos=0.7}, shift left=2, shorten <=5pt, shorten >=4pt, Rightarrow, from=3, to=1]
	\arrow["\cong"{description, pos=0.3}, draw=none, from=2-1, to=3]
\end{tikzcd} 
$=$ 
}
\end{center}
\begin{center}
\scalebox{0.8}{
\hspace{-1cm}\begin{tikzcd}[ampersand replacement=\&]
	\& {\cc_{\pschG}} \\
	{\cc_{\pschG^{FG}}} \& {\cc_{\pschG^{HK}}} \&\&\& {\gk(S,T)[p,q]} \\
	\& {\cc_{\pschG^{FG,HK}}}
	\arrow["\labL", curve={height=-12pt}, from=1-2, to=2-5]
	\arrow["{J_{FG}}"', from=1-2, to=2-1]
	\arrow["{J_{HK}^{FG}}"', from=2-1, to=3-2]
	\arrow[""{name=0, anchor=center, inner sep=0}, "{\labL^{\alpha',\beta'}_{FG,HK}}"', curve={height=24pt}, from=3-2, to=2-5]
	\arrow[""{name=1, anchor=center, inner sep=0}, "{\labL^{\beta'}_{FG,HK}}"{description, pos=0.3}, curve={height=6pt}, from=3-2, to=2-5]
	\arrow[""{name=2, anchor=center, inner sep=0}, "{\labL_{FG,HK}}"{description, pos=0.3}, curve={height=-12pt}, from=3-2, to=2-5]
	\arrow["{J_{HK}}"{description}, from=1-2, to=2-2]
	\arrow["{J_{FG}^{HK}}"{description}, from=2-2, to=3-2]
	\arrow["\cong"{description, pos=0.3}, draw=none, from=2-1, to=2-2]
	\arrow[""{name=3, anchor=center, inner sep=0}, "{\labL_{HK}}"{description}, curve={height=-12pt}, from=2-2, to=2-5]
	\arrow["\cong"{description, pos=0.3}, draw=none, from=2-2, to=2-5]
	\arrow["{{\Pi^\ast}}"{pos=0.6,yshift=-0.05cm,xshift=0.05cm}, shift left=3, shorten <=4pt, shorten >=4pt, Rightarrow, from=1, to=0]
	\arrow["{{\Omega^\ast}}"{pos=0.7}, shift left=2, shorten <=5pt, shorten >=4pt, Rightarrow, from=2, to=1]
	\arrow["\cong"{description}, draw=none, from=1-2, to=3]
\end{tikzcd} $=$
\begin{tikzcd}[ampersand replacement=\&]
	\& {\cc_{\pschG}} \\
	{\cc_{\pschG^{FG}}} \& {\cc_{\pschG^{HK}}} \&\&\& {\gk(S,T)[p,q]} \\
	\& {\cc_{\pschG^{FG,HK}}}
	\arrow["\labL", curve={height=-12pt}, from=1-2, to=2-5]
	\arrow["{J_{FG}}"', from=1-2, to=2-1]
	\arrow["{J_{HK}^{FG}}"', from=2-1, to=3-2]
	\arrow[""{name=0, anchor=center, inner sep=0}, "{\labL^{\alpha',\beta'}_{FG,HK}}"', curve={height=24pt}, from=3-2, to=2-5]
	\arrow[""{name=1, anchor=center, inner sep=0}, "{\labL^{\beta'}_{FG,HK}}"{description, pos=0.3}, curve={height=6pt}, from=3-2, to=2-5]
	\arrow["{J_{HK}}"{description}, from=1-2, to=2-2]
	\arrow[""{name=2, anchor=center, inner sep=0}, "{J_{FG}^{HK}}"{description}, from=2-2, to=3-2]
	\arrow["\cong"{description, pos=0.3}, draw=none, from=2-1, to=2-2]
	\arrow[""{name=3, anchor=center, inner sep=0}, "{\labL_{HK}}"{description, pos=0.3}, curve={height=-12pt}, from=2-2, to=2-5]
	\arrow[""{name=4, anchor=center, inner sep=0}, "{\labL_{HK}^{\beta'}}"{description, pos=0.3}, curve={height=6pt}, from=2-2, to=2-5]
	\arrow["{{\Pi^\ast}}"{pos=0.6,yshift=-0.05cm,xshift=0.05cm}, shift left=3, shorten <=4pt, shorten >=4pt, Rightarrow, from=1, to=0]
	\arrow["\cong"{description}, draw=none, from=1-2, to=3]
	\arrow["\cong"{description, pos=0.4}, draw=none, from=2, to=1]
	\arrow["{{\Omega^\ast}}"{xshift=0.1cm}, shift left=2, shorten <=4pt, shorten >=4pt, Rightarrow, from=3, to=4]
\end{tikzcd}
} 
\end{center}

The first and last equality follow by \Cref{lem:alpha'-and-FG} applied to $(\Gamma^\ast,HK)$ and $(\Omega^\ast,FG)$, and the second equality follows by part (iii) of \Cref{rmk:lab-and-FG-HK} and \Cref{thm:easy-gen-loc-int} for the 3-cells $\Pi,\Omega$ and the labelling $\labL_{FG,HK}$.  
\end{proof}


Now, we will briefly explain how to understand/use this theorem. Let us consider an object $\overline{K}\in\cc_\pschG$, then $\labL(\overline{K})$ is the corresponding composite of the 2-dimensional pasting diagram defined by the labelling. Since $\Pi$ and $\Omega$ act on two faces of the labelling, we cannot apply the 3-cells to any composite. The pseudo-inverses $J$ give choices of different composites to which we can apply $\Pi$ and/or $\Omega$. Below, setting $X=\labL(\overline{K})$,  we write these choices with $X_{FG}=\labL_{FG}J(\overline{K})$, $X_{FG}^{\alpha'}=\labL^{\alpha'}_{FG}J(\overline{K})$ and analogous for $\beta'$ and $HK$. 
Then the equation in \Cref{thm:hard-gen-loc-int} is the same as the commutativity of the following diagram. 
    \begin{equation}\label{eq:gen-hard--local-interch}
\begin{tikzcd}[ampersand replacement=\&]
	X \& {X_{FG}} \& {X_{FG}^{\alpha'}} \& {{(X_{FG}^{\alpha'})}_{HK}} \\
	{X_{HK}} \\
	{X_{HK}^{\beta'}} \& {{(X_{HK}^{\beta'})}_{FG}} \& {{(X_{HK}^{\beta'})}_{FG}^{\alpha'}} \& {{(X_{FG}^{\alpha'})}_{HK}^{\beta'}}
	\arrow["\gamma"', from=1-1, to=2-1]
	\arrow["\gamma"', from=3-1, to=3-2]
	\arrow["{\Pi^\ast}", from=1-2, to=1-3]
	\arrow["\gamma", from=1-3, to=1-4]
	\arrow["{\Omega^\ast}"', from=2-1, to=3-1]
	\arrow["{\Pi^\ast}"', from=3-2, to=3-3]
	\arrow["\gamma", from=1-1, to=1-2]
	\arrow["{\Omega^\ast}", from=1-4, to=3-4]
	\arrow["\gamma"', from=3-3, to=3-4]
	\arrow["{\overline{\Omega^\ast}}"', color={rgb,255:red,214;green,92;blue,92}, curve={height=20pt}, from=1-1, to=3-1]
	\arrow["{\overline{\Pi^\ast}}"', color={rgb,255:red,92;green,214;blue,92}, curve={height=15pt}, from=1-1, to=1-3]
	\arrow["{\overline{\Omega^\ast}}"', color={rgb,255:red,214;green,92;blue,92}, curve={height=-10pt}, from=1-3, to=3-4]
	\arrow["{\overline{\Pi^\ast}}", color={rgb,255:red,92;green,214;blue,92}, curve={height=-15pt}, from=3-1, to=3-3]
\end{tikzcd}
\end{equation}

We underline that, since we proved the theorem above for any choice of pseudo-inverses, the diagram \eqref{eq:gen-hard--local-interch} is commutative for any choice of $X_{FG}, {(X_{FG})}_{HK}, X_{HK}$ etc. This theorem allow us to do calculations regarding 3-cells in a Gray-category using the 2-dimensional pasting diagrams.  

\subsection{Special Case}
\label{sec:2-1-special-case}

We end considering when the composition is not unique, i.e. a pasting scheme satisfying the hypothesis of \Cref{lem:F-G-composable} part (a). We will consider a leading example, which does not comprise every case but provides the strategy to use each time. Let us consider the following pasting scheme: 
\begin{center}
$\pschG:=\quad$
\begin{tikzcd}
	S & A & B & C & D & T.
	\arrow[""{name=0, anchor=center, inner sep=0}, "{a_0}", curve={height=-12pt}, from=1-1, to=1-2]
	\arrow[""{name=1, anchor=center, inner sep=0}, "{a_1}"', curve={height=12pt}, from=1-1, to=1-2]
	\arrow["f"{description}, from=1-2, to=1-3]
	\arrow[""{name=2, anchor=center, inner sep=0}, "{c_0}", curve={height=-12pt}, from=1-5, to=1-6]
	\arrow[""{name=3, anchor=center, inner sep=0}, "{c_1}"', curve={height=12pt}, from=1-5, to=1-6]
	\arrow[""{name=4, anchor=center, inner sep=0}, "{b_0}", curve={height=-12pt}, from=1-3, to=1-4]
	\arrow[""{name=5, anchor=center, inner sep=0}, "{b_1}"', curve={height=12pt}, from=1-3, to=1-4]
	\arrow["g"{description}, from=1-4, to=1-5]
	\arrow["F"{description}, draw=none, from=0, to=1]
	\arrow["H"{description}, draw=none, from=4, to=5]
	\arrow["G"{description}, draw=none, from=2, to=3]
\end{tikzcd}
\end{center}
The groupoid $\cc_\pschG$ has objects such as $FGH$ and $HFG$ but also $FHG$. \black The point is that \emph{we have a choice} of the composition where the labelling of $F$ and $G$ occur next to each other. In other words, we can replace the previous pasting scheme with one of the following: 
\begin{center}
    \begin{tabular}{cc}
        $HFG\leftrightsquigarrow\quad$ &  
\begin{tikzcd}
	S & A & B & C & D & T
	\arrow["{a_0}", from=1-1, to=1-2]
	\arrow["f", from=1-2, to=1-3]
	\arrow[""{name=0, anchor=center, inner sep=0}, from=1-3, to=1-4]
	\arrow["g", from=1-4, to=1-5]
	\arrow[""{name=1, anchor=center, inner sep=0}, "{b_0}", curve={height=-18pt}, from=1-3, to=1-4]
	\arrow["{c_0}", from=1-5, to=1-6]
	\arrow[""{name=2, anchor=center, inner sep=0}, "{c_1gb_1fa_1}"', curve={height=24pt}, from=1-1, to=1-6]
	\arrow["{F\#G}"{description}, draw=none, from=0, to=2]
	\arrow["H"{description}, draw=none, from=1, to=0]
\end{tikzcd}\\
      $FGH\leftrightsquigarrow\quad$   &  
\begin{tikzcd}[ampersand replacement=\&]
	S \& A \& B \& C \& D \& T
	\arrow["{a_1}"', from=1-1, to=1-2]
	\arrow["f"', from=1-2, to=1-3]
	\arrow[""{name=0, anchor=center, inner sep=0}, from=1-3, to=1-4]
	\arrow["g"', from=1-4, to=1-5]
	\arrow[""{name=1, anchor=center, inner sep=0}, "{b_1}"', curve={height=18pt}, from=1-3, to=1-4]
	\arrow["{c_1}"', from=1-5, to=1-6]
	\arrow[""{name=2, anchor=center, inner sep=0}, "{c_0gb_0fa_0}", curve={height=-24pt}, from=1-1, to=1-6]
	\arrow["{F\#G}"{description}, draw=none, from=0, to=2]
	\arrow["H"{description}, draw=none, from=1, to=0]
\end{tikzcd}
    \end{tabular}
\end{center}

Both scenarios lead us back to the situations we examine in \Cref{lem:F-G-composable}. 



\bibliography{References}
\bibliographystyle{alpha}

\end{document}